\documentclass[11pt]{article}

\usepackage[letterpaper,margin=1in]{geometry}

\usepackage{bm}
\usepackage{amsmath,amssymb,amsthm,mathtools}
\usepackage{xparse}

\usepackage{algorithm}
\usepackage{algpseudocode}
\usepackage{tikz}
\usetikzlibrary{arrows.meta,positioning,decorations.pathreplacing}

\usepackage{enumitem}
\usepackage[most]{tcolorbox}
\usepackage{authblk}

\usepackage[round,authoryear]{natbib}
\usepackage[colorlinks=true,linkcolor=blue,citecolor=blue,urlcolor=blue]{hyperref}
\usepackage[nameinlink,capitalize]{cleveref}

\allowdisplaybreaks

\newtheorem{theorem}{Theorem}
\newtheorem{lemma}{Lemma}

\theoremstyle{definition}

\theoremstyle{remark}

\newtcblisting{promptbox}[1][]{
breakable,
enhanced,
listing only,
listing options={
basicstyle=\normalfont,
breaklines=true,
breakautoindent=false,
breakindent=0pt,
postbreak={},
columns=fullflexible,
keepspaces=true,
aboveskip=0pt,
belowskip=0pt,
literate={-}{{-\allowbreak}}1
},
colback=white,
colframe=gray!35,
boxrule=0.4pt,
arc=2pt,
left=6pt,
right=6pt,
top=6pt,
bottom=6pt,
title={Prompt},
fonttitle=\bfseries,
coltitle=black,
colbacktitle=gray!5,
titlerule=0.4pt,
lefttitle=6pt,
righttitle=6pt,
toptitle=4pt,
bottomtitle=4pt,
#1
}

\newtcolorbox{chatgptbox}[1][]{
breakable,
enhanced,
colback=gray!8,
colframe=gray!35,
boxrule=0.4pt,
arc=2pt,
left=6pt,
right=6pt,
top=6pt,
bottom=6pt,
title={ChatGPT-generated response},
fonttitle=\bfseries,
coltitle=black,
colbacktitle=gray!15,
titlerule=0.4pt,
lefttitle=6pt,
righttitle=6pt,
toptitle=4pt,
bottomtitle=4pt,
fontupper=\normalfont,
before upper={
\setlength{\parindent}{0pt}
\setlength{\parskip}{0.75\baselineskip}
},
#1
}

\renewcommand{\P}[1]{\mathbb{P}\left(#1\right)}
\NewDocumentCommand{\E}{o m}{%
  \ensuremath
  {\mathbb{E}%
  \IfValueT{#1}{_{#1}}%
  \!\left[#2\right]}
}

\newcommand{\given}{\;\middle|\;}

\newcommand{\Zp}{\mathbb Z_+}
\newcommand{\ql}{q_{\textup{L}}}
\newcommand{\qh}{q_{\textup{H}}}
\newcommand{\xl}{x_{\textup{L}}}
\newcommand{\xh}{x_{\textup{H}}}
\newcommand{\K}{\mathcal K}
\newcommand{\Qtrue}{\mathcal Q}
\newcommand{\Qop}{\widetilde{\mathcal Q}}
\newcommand{\Pop}{\mathcal P}
\newcommand{\Aop}{\mathcal A}
\newcommand{\DFop}{\mathcal{D}_{\textup{F}}}
\newcommand{\DSop}{\mathcal{D}_{\textup{S}}}
\newcommand{\Bellman}{\mathcal{T}}
\newcommand{\Bellmanmodified}{\widetilde{\mathcal{T}}}

\newcommand{\1}{\mathbf 1}
\newcommand{\piopt}{\pi^*}
\newcommand{\piavgopt}{\overline{\pi}^*}
\newcommand{\Val}{V_{\alpha}^{(\infty)}}
\newcommand{\Dzo}{D_{01}^*}
\newcommand{\Dtw}{D_{12}^*}
\newcommand{\tpi}{\widetilde{\pi}}
\newcommand{\tC}{\widetilde{C}}

\newif\ifarxiv
\arxivtrue
\ifarxiv

\ifarxiv
\newenvironment{myproof}[1][Proof]{%
\begin{proof}[#1]%
}{%
\end{proof}%
}
\else
\newenvironment{myproof}[1][Proof.]{%
\begin{proof}{#1}%
}{%
\end{proof}%
}
\fi

\title{Optimality of a Threshold Policy for a Queueing System with One Fast Server and Two Identical Slow Servers}

\author[1]{Weina Wang}
\author[2]{Taha Ameen}
\author[3]{Yudong Chen}
\author[1]{Yige Hong}
\author[1]{Josh Nichols}
\author[3]{Matthew Zurek}

\affil[1]{Computer Science Department, Carnegie Mellon University\newline
\texttt{weinaw@cs.cmu.edu}, \texttt{yigeh@andrew.cmu.edu},
\texttt{joshuani@andrew.cmu.edu}}
\affil[2]{Department of Electrical and Computer Engineering \& Coordinated Science Lab,
University of Illinois\newline
\texttt{tahaa3@illinois.edu}}
\affil[3]{Department of Computer Sciences, University of Wisconsin-Madison\newline
\texttt{yudongchen@cs.wisc.edu},
\texttt{matthew.zurek@wisc.edu}
}
\date{}

\begin{document}

\maketitle

\begin{abstract}
This paper studies the optimal control problem of a queueing system with three servers: one fast server and two identical slow servers.
The two-server version of this problem, with one fast server and one slow server, was introduced by \citet{LinKum_84}, and the optimal policy has been shown to be of threshold type.
However, generalizing this result beyond the two-server setting has been considered an open problem.
In this paper, we resolve the first nontrivial case by proving that a threshold policy is optimal for the three-server system considered.

The core technical ideas in this paper are generated by GPT-5.5 Pro.
\ifarxiv
We have included a short report describing the authors' interactions with GPT-5.5 Pro.
\fi
The authors verified the proofs and rewrote the paper for better rigor, clarity, and exposition.
In addition, three key lemmas have also been verified in Lean 4.
\end{abstract}

\section{Introduction}\label{sec:intro}

We consider an optimal control problem for a single-queue system with one fast server and two identical slow servers, as illustrated in Figure~\ref{fig:model}.
Jobs arrive according to a Poisson process with rate~$\lambda$.
Server~1 is the fast server and has exponentially distributed service times with rate~$\mu_1$;
Servers~2 and 3 are the identical slow servers, and they have exponentially distributed service times with rate $\mu_2$, where
\begin{equation}
    0<\mu_2<\mu_1.
\end{equation}
The control policy decides when to assign queued jobs to idle servers, and the decisions are made right after a job arrival or a job departure.
Job assignments are nonpreemptive, i.e., once a job is assigned to a server, it remains there until service completion.
The goal of the optimal control is to minimize the cost, where the cost is the total number of jobs in the system.
We study both the infinite-horizon discounted-cost criterion and the infinite-horizon average-cost criterion under the stability condition $\lambda<\mu_1+2\mu_2$.
\ifarxiv
\begin{figure}
\centering
\includegraphics[width=0.4\linewidth]{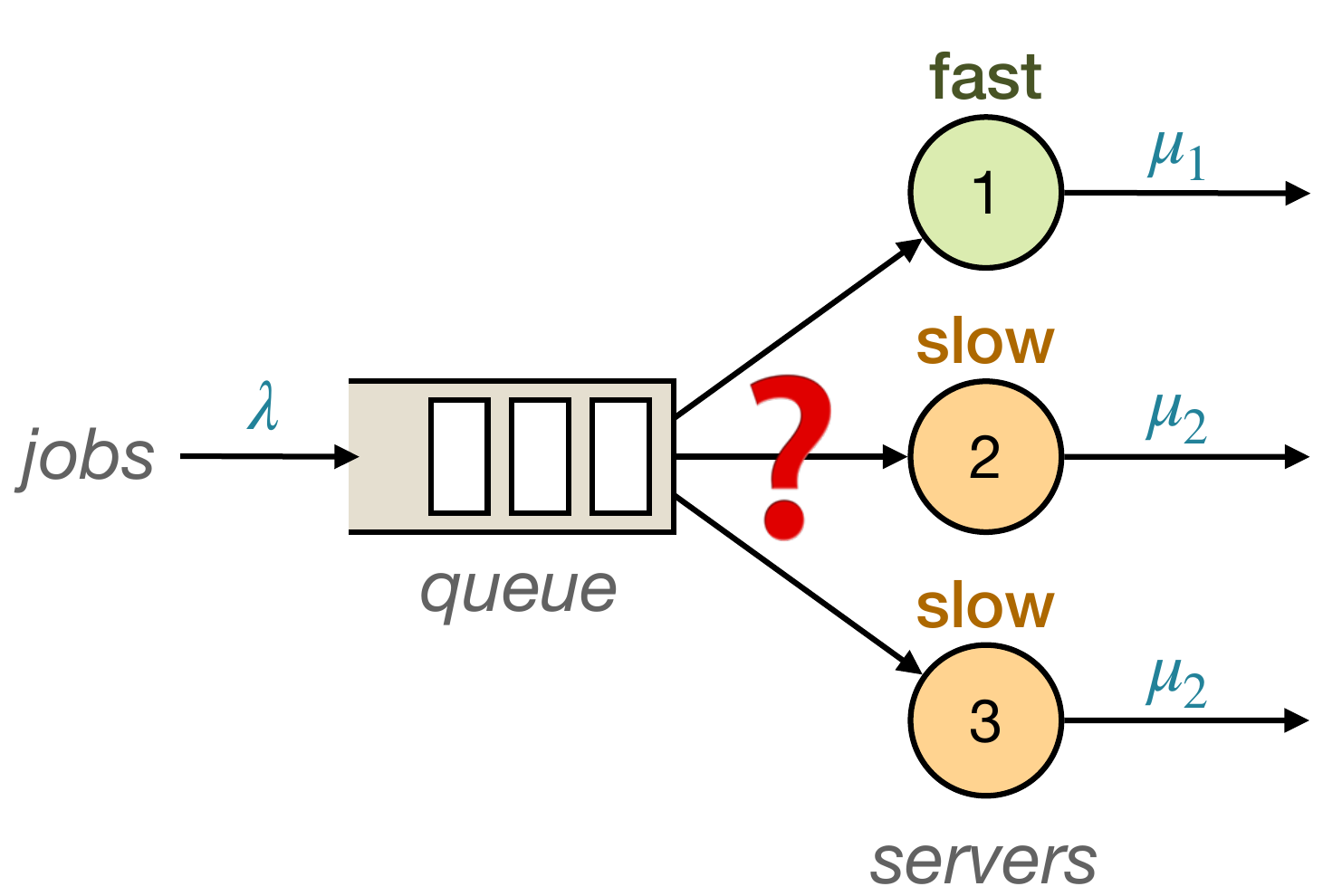}
\caption{A queueing system with one fast server and two identical slow servers.
We assume $0<\mu_2<\mu_1$, so Server 1 is the fast server and Servers 2 and 3 are the slow servers.}
\label{fig:model}
\end{figure}
\else
\begin{figure}
\FIGURE
{\includegraphics[width=0.4\linewidth]{stsy/figures/model.pdf}}
{A queueing system with one fast server and two identical slow servers. 
\label{fig:model}}
{We assume $0<\mu_2<\mu_1$, so Server 1 is the fast server and Servers 2 and 3 are the slow servers.}
\end{figure}
\fi

The two-server version of this problem, with one fast server and one slow server, was introduced by \citet{LinKum_84}.
They proved that there exists an optimal policy that is of threshold type: the fast server is kept busy whenever possible, and the slow server is used only when the queue length exceeds a threshold.
The proof is based on policy iteration.
This problem has since been revisited by different methods, including the probabilistic coupling arguments of \citet{Wal_84}, and the value-iteration arguments of \citet{Koo_95}, and variants of the problem have also been studied in asymptotic regimes (see, e.g., \citealt{TehWar_02,Arm_05}).
It has been conjectured that a similar threshold structure should hold for systems with more than two servers.
However, as explained by \citet{Koo_22} in the special issue ``100 Views on Queues'' celebrating Volume 100 of \emph{Queueing Systems}, this generalization is delicate; some prior attempts were later found to be incomplete \citep{deVZho_06}.
Therefore, this problem with more than two servers is still considered open.
Koole concludes his note with the challenge: ``This problem is too fundamental to be left open'' \citep{Koo_22}.

This paper proves the threshold conjecture for the first nontrivial open case beyond the two-server setting: one fast server and two identical slow servers.
Three key lemmas in this paper have been verified in Lean 4 \citep{MouUll_21},\footnote{\url{https://lean-lang.org/}}
with code available at
\ifarxiv
the GitHub repository \citet{GitLeanOneFastTwoSlow_26}.
\else
an anonymous GitHub repository \citet{GitLeanOneFastTwoSlow_26_anonymous}.
\fi

\subsection*{Disclaimer}
The core technical ideas in this paper are generated by GPT-5.5 Pro.
The authors verified the proofs and rewrote the paper for better rigor, clarity, and exposition.
Thus, while the paper is written in the usual first-person plural style, the reader should be aware that the initial discovery of the main proof ideas was by AI.
\ifarxiv
We have included a short report describing the authors' interactions with GPT-5.5 Pro in Appendix~\ref{sec:report-gpt}.
\fi

\section{Result and Approach}\label{sec:result}

\begin{theorem}\label{thm:main}
For the considered optimal control problem for a queueing system with one fast server and two identical slow servers, there is an optimal stationary policy under the discounted-cost criterion with the following threshold structure.
Let $q$ denote the number of queued jobs at the decision time.
It suffices to consider the case $q\ge 1$.
\begin{itemize}
\item \textbf{Fast Server:} When the fast server is idle, it is assigned one queued job.
\item \textbf{Slow Servers:} There exist thresholds $\ql$ and $\qh$, with $0\le \ql\le \qh\le \infty$, such that the job assignment decisions for slow servers are as follows.
\begin{itemize}
    \item When both slow servers are idle, the number of jobs assigned to them is
    \begin{equation}
        \begin{cases}
        0, & \text{if } q<\ql,\\
        1, & \text{if } \ql\le q<\qh,\\
        \min\{2,q\}, & \text{if } q\ge \qh.
        \end{cases}
    \end{equation}
    \item When one slow server is idle, one job is assigned to it if $q\ge \qh-1$.
\end{itemize} 
\end{itemize}

If $\lambda<\mu_1+2\mu_2$, there is an average-cost optimal policy with the same structure.
\end{theorem}

\paragraph{Remark.}
We can understand the two thresholds $\ql$ and $\qh$ in the following way.
When both slow servers are idle, the job assignment decision can be thought of as consisting of two sequential decisions.
We first decide whether to raise the number of jobs at slow servers from $0$ to $1$, which is determined based on the lower threshold~$\ql$.
Conditional on raising to~$1$, the queue length~$q$ has decreased by~$1$ and become $q-1$, and the decision whether to further raise the number of jobs at slow servers to~$2$ is the same decision as when one slow server is initially idle.
Therefore, the threshold appears as $\qh$ in the two-initially-idle case and as $\qh-1$ in the one-initially-idle case.

\paragraph{Proof Approach.}
The proof in this paper is based on value iteration, a common approach in the literature (see, e.g., \citealt{Haj_84,Koo_95}).
We propagate a set of six inequalities through value iteration.
These inequalities characterize monotonicity, convexity, and supermodularity properties of the value function.
The main novelty, relative to the proof for the two-server problem using value iteration \citep{Koo_95}, is a convexity property \eqref{eq:convex-y-fix-total} in the number of jobs at slow servers when fixing the total number of jobs in the system.
The main role of this property is to guarantee that the two thresholds are ordered, i.e., $\ql\le \qh$.

\paragraph{Basic Notation.}
We use $\mathbb{R}$ to denote the set of real numbers, and $\Zp$ to denote the set of nonnegative integers.
For any real number $u$, let $u^+\triangleq \max\{u,0\}$.

\paragraph{Organization of the Paper.}
The rest of the paper is organized as follows.
We first focus on the discounted-cost problem in Sections~\ref{sec:prelim} and \ref{sec:discounted}.
In Section~\ref{sec:prelim}, we set up the state representation, value iteration, and define a cone of value functions with six desired inequalities.
Section~\ref{sec:discounted} is our main section, where we prove the result for the discounted-cost problem in Theorem~\ref{thm:main}.
We then study the average-cost setting in Section~\ref{sec:average}, where we extend the result from the discounted-cost setting to the average-cost setting using a standard approach.
\ifarxiv
In addition to the technical content, Appendix~\ref{sec:report-gpt} includes a short report on the authors' interactions with GPT-5.5 Pro.
\fi

\section{Preliminaries for the Discounted-Cost Problem}\label{sec:prelim}
This section prepares for the proof of Theorem~\ref{thm:main} for the discounted-cost problem.
We first set up the state representation in Section~\ref{sec:state} and the value iteration in Section~\ref{sec:value-iteration}.
Then in Section~\ref{sec:cone}, we introduce a cone $\K$ of value functions, defined using a set of six inequalities, which are the inequalities to be propagated through value iteration.
We also show that for an optimal policy to have the threshold structure in Theorem~\ref{thm:main}, it suffices to show that the optimal value function is in the cone $\K$.

\subsection{State Representation}\label{sec:state}

The state of the system can be represented by
\begin{equation}\label{eq:state-full}
    s = (q, z, y) \in \Zp\times\{0,1\}\times\{0,1,2\},
\end{equation}
where $q$ is the number of queued jobs, $z$ is the number of jobs at the fast server, and $y$ is the number of jobs at the slow servers.
The cost is the total number of jobs in the system $q+z+y$.

When we consider the post-decision state, the state representation for the discounted-cost problem can be simplified using the following lemma.
\begin{lemma}\label{lem:simplified-state}
An optimal policy for the discounted-cost problem never lets the fast server idle after a decision if there are jobs in the queue at the decision time.
\end{lemma}

Lemma~\ref{lem:simplified-state} also establishes the structure for the fast server for the discounted-cost problem in Theorem~\ref{thm:main}.
The proof of Lemma~\ref{lem:simplified-state} is similar to that for the two-server problem in \citet{Wal_84}, and thus it is deferred to Appendix~\ref{sec:proof-lem-simplified-state}.

With Lemma~\ref{lem:simplified-state}, instead of tracking $q$ and $z$, i.e., the number of jobs in the queue and at the fast server, respectively, it suffices to track their sum $x\triangleq q+z$.
Specifically, if $x=0$, the fast server is idle and the queue is empty; if $x\ge 1$, there is one job at the fast server and $x-1$ jobs in the queue.
With slight abuse of notation by reusing $s$, the system state can be written as
\begin{equation}
    s = (x,y)\in\Zp\times\{0,1,2\}.
\end{equation}
Then the cost function is
\begin{equation}
    c(x,y)=x+y.
\end{equation}

\subsection{Value Iteration Setup}\label{sec:value-iteration}

We consider the post-decision value function. 
That is, right after each event, i.e., a job arrival or departure, we first make the job assignment decisions, and then observe the state $s=(x,y)$.
Under a policy $\pi$ and discount rate $\alpha>0$, the value function, which is the total discounted cost, is given by
\begin{equation}
    V^{\pi}_{\alpha}(s)=\E{\int_0^{\infty} e^{-\alpha t}\left(X(t)+Y(t)\right)dt \given S(0)=s},
\end{equation}
where $S(t)=(X(t),Y(t))$ is the post-decision state at time $t$.
This cost is finite since jobs arrive according to a Poisson process.
Let $V^*_{\alpha}$ be the optimal discounted value function.

\paragraph{Bellman Operator.}
To define the Bellman operator for this post-decision value function, we first introduce the following two operators defined on real-valued functions on the state space $\Zp\times\{0,1,2\}$.
\begin{itemize}
\item \emph{Job assignment operator $\Qtrue$.\phantom{a}} 
For any function $f\colon \Zp\times\{0,1,2\}\to\mathbb
R$, let
\begin{equation}
(\Qtrue f)(x,y)=\min_{y\le k\le \min\{2,(x-1)^++y\}} f(x+y-k,k),  
\label{eq:old-1-2}
\end{equation}
where recall $(x-1)^+\triangleq \max\{x-1,0\}$.
When $f$ is interpreted as a value function, given the current state $(x,y)$, the operator $\Qtrue$ determines the number of jobs $k$ at the slow servers to minimize the value.
The constraint $k\ge y$ is due to the fact that existing jobs at the slow servers cannot be removed since job assignments are nonpreemptive.
The constraint $k\le (x-1)^+ + y$ means that at most $(x-1)^+$ additional jobs can be assigned to the slow servers.
This is again because job assignments are nonpreemptive;
only queued jobs can be assigned, while the job at the fast server, if any, cannot be moved.
In particular, some boundary cases are given below for clarity
\begin{equation}
    (\Qtrue f)(0,y)=f(0,y),\quad
    (\Qtrue f)(1,y)=f(1,y),\quad
    (\Qtrue f)(2,0)=\min\{f(2,0),f(1,1)\}.
\end{equation}

\item \emph{Event operator $\Pop$.\phantom{a}}
The event operator $\Pop$ captures the expected value after a job arrival or departure.
For any function $f\colon \Zp\times\{0,1,2\}\to\mathbb
R$, let
\begin{align}
(\Pop f)(x,0)&=\lambda f(x+1,0)+\mu_1 f((x-1)^+,0)+2\mu_2 f(x,0),\\
(\Pop f)(x,1)&=\lambda f(x+1,1)+\mu_1 f((x-1)^+,1)+\mu_2 f(x,0)+\mu_2 f(x,1),\\
(\Pop f)(x,2)&=\lambda f(x+1,2)+\mu_1 f((x-1)^+,2)+2\mu_2 f(x,1).
\end{align}
Here, the terms with $\lambda$ correspond to a job arrival, the terms with $\mu_1$ correspond to a potential job departure from the fast server, and the terms with $\mu_2$ correspond to potential job departures from the slow servers.
Let $\gamma=\lambda+\mu_1+2\mu_2$ be the total event rate.
Then note that $\frac{1}{\gamma}\Pop$ corresponds to a transition probability matrix.
\end{itemize}

We are now ready to define the Bellman operator $\Bellman$.
For any function $f\colon \Zp\times\{0,1,2\}\to\mathbb
R$, let
\begin{equation}\label{eq:Bellman}
(\Bellman f)(x, y) = c(x,y) + \frac{\rho_{\alpha}}{\gamma} (\Pop \Qtrue f)(x,y),
\end{equation}
where $\rho_{\alpha}=\frac{\gamma}{\alpha+\gamma}<1$.
Using standard arguments (see, e.g., \citealt{LinKum_84,Haj_84}), one can show that the following value iteration
\begin{equation}
    V_\alpha^{(n+1)}=\Bellman V_\alpha^{(n)},
\end{equation}
converges to $\Val\triangleq (\alpha+\gamma)V_\alpha^*$.
We will ignore the constant factor $\alpha+\gamma$ and treat $\Val$ as the optimal value function.

Standard results for discounted Markov decision processes imply that any policy attaining the minimum in $\Qtrue \Val$ is an optimal policy.
Specifically, for any deterministic policy $\pi$, let $\pi(x,y)$ denote the number of jobs assigned to the slow servers given the post-event state $(x,y)$.
Define the policy $\pi^*$ by
\begin{equation}
    \pi^*(x,y)=\min\;\operatorname*{arg\;min}\limits_{y\le k\le\min\{2,(x-1)^+ +y\}} \Val(x+y-k,k),
    \label{eq:opt-policy}
\end{equation}
where the outer $\min$ is for tie breaking.
Then $\pi^*$ is an optimal policy.

\paragraph{Modified Bellman Operator.}
In the proof, instead of using the original job assignment operator $\Qtrue$ and the corresponding Bellman operator $\Bellman$, we use modified versions of them, where we allow the job at the fast server to be reassigned to slow servers.
However, we later prove that such actions will not be used after one step of the modified value iteration (Lemma~\ref{lem:vi-consistency}).
Intuitively, moving a job from the fast server to a slow server can only increase the cost, and thus will not be used even when allowed.
Specifically, we define the modified job assignment operator $\Qop$ as
\begin{equation}
(\Qop f)(x,y)=
\min_{y\le k\le \min\{2,x+y\}} f(x+y-k,k),
\label{eq:old-1-3}
\end{equation}
which uses $x+y$ in the upper bound for $k$ instead of the $(x-1)^++y$ in $\Qtrue$ in \eqref{eq:old-1-2}.
Then the corresponding modified Bellman operator is defined as
\begin{equation}
(\Bellmanmodified f)(x, y) = c(x,y) + \frac{\rho_{\alpha}}{\gamma} (\Pop \Qop f)(x,y).
\end{equation}
One can still show that the value iteration
\begin{equation}\label{eq:vi-modified}
    \widetilde{V}_{\alpha}^{(n+1)}=\Bellmanmodified \widetilde{V}_{\alpha}^{(n)},
\end{equation}
converges to a limit $\widetilde{V}_{\alpha}^{(\infty)}$.
We will prove that $\widetilde{V}_{\alpha}^{(\infty)}=\Val$ in Lemma~\ref{lem:vi-consistency} in Section~\ref{sec:discounted}.

\subsection{Invariant Cone}\label{sec:cone}

We introduce a cone $\K$ below, which is designed in a way such that if the optimal value function is in $\K$, then the corresponding optimal policy has the threshold structure in Theorem~\ref{thm:main}.
Let cone $\K$ be a set of functions $f\colon \Zp\times\{0,1,2\}\to \mathbb{R}$ that satisfy the following six properties.
\begin{align}
&\textbf{(K1) Monotonicity in $x$:} \notag\\[-.4em]
&\qquad f(x+1,y)-f(x,y)\ge0,
    &&y=0,1,2,\quad x\ge0.                         \tag{K1} \label{eq:mono-x}\\
&\textbf{(K2) First multimodularity inequality:} \notag\\[-.4em]
&\qquad f(x+2,y)+f(x,y+1)-f(x+1,y)-f(x+1,y+1)\ge0,
    &&y=0,1,\quad x\ge0.                           \tag{K2}\label{eq:multimodular-1}\\
&\textbf{(K3) Second multimodularity inequality:} \notag\\[-.4em]
&\qquad f(x,y)+f(x+1,y+1)-f(x+1,y)-f(x,y+1)\ge0,
    &&y=0,1,\quad x\ge0.                           \tag{K3}\label{eq:multimodular-2}\\
&\textbf{(K4) Monotonicity in $y$:} \notag\\[-.4em]
&\qquad f(x,y+1)-f(x,y)\ge0,
    &&y=0,1,\quad x\ge0.                           \tag{K4}\label{eq:mono-y}\\
&\textbf{(K5) Convexity in $y$ fixing $x$:} \notag\\[-.4em]
&\qquad f(x,0)+f(x,2)-2f(x,1)\ge0,
    &&x\ge0.                                       \tag{K5}\label{eq:convex-y-fix-x}\\
&\textbf{(K6) Convexity in $y$ fixing total population:} \notag\\[-.4em]
&\qquad f(x+2,0)+f(x,2)-2f(x+1,1)\ge0,
    &&x\ge0.                                       \tag{K6}\label{eq:convex-y-fix-total}
\end{align}
It is easy to verify that a constant function, e.g., the zero function, is in $\K$.
Since these inequalities are linear and closed under pointwise limits, it is also easy to see that $\K$ is closed under nonnegative linear combinations and pointwise limits. 

The first three properties, \eqref{eq:mono-x}--\eqref{eq:multimodular-2}, are analogues of the corresponding properties in the two-server problem in \cite{Koo_95}.
It is worth noting that as in the two-server problem, the two multimodularity inequalities \eqref{eq:multimodular-1} and \eqref{eq:multimodular-2} imply convexity in $x$:
\begin{align}
f(x+2,y)+f(x,y)-2f(x+1,y)\ge0,
\qquad y=0,1,2,\quad x\ge0.            \label{eq:convex-x}
\end{align}
Indeed, for $y=0,1$, \eqref{eq:convex-x} follows by adding \eqref{eq:multimodular-1} and \eqref{eq:multimodular-2} at $(x,y)$.
For $y=2$, \eqref{eq:convex-x} follows by adding \eqref{eq:multimodular-1} at $(x,1)$ and \eqref{eq:multimodular-2} at $(x+1,1)$.

The remaining three inequalities have not appeared in the two-server problem.
Property \eqref{eq:mono-y} is a natural monotonicity property in $y$, and properties \eqref{eq:convex-y-fix-x} and \eqref{eq:convex-y-fix-total} are convexity properties in $y$, fixing $x$ and fixing the total population $x+y$, respectively.
The property \eqref{eq:convex-y-fix-total} is particularly novel, and it plays an especially important role in the proof.

We establish the threshold structure of the optimal policy in Lemma~\ref{lem:threshold}, assuming that the optimal value function is in the cone $\K$.
The proof of Lemma~\ref{lem:threshold} shows that properties \eqref{eq:multimodular-1} and \eqref{eq:convex-y-fix-total} lead to the desired threshold structure.

\subsubsection*{Intuitions on \eqref{eq:multimodular-1} and \eqref{eq:convex-y-fix-total}.}
Here we first provide some intuitive explanation on how \eqref{eq:multimodular-1} and \eqref{eq:convex-y-fix-total} lead to the threshold structure, which is the key intuition in this paper.
Suppose the optimal value function $\Val\in\K$.
Then applying \eqref{eq:multimodular-1} gives
\begin{equation}
\Val(x+1,0)-\Val(x,1)
\ge \Val(x,0)-\Val(x-1,1).
\end{equation}
The difference $\Val(x,0)-\Val(x-1,1)$ on the right hand side compares two actions when both slow servers are idle: assigning no job to a slow server leads to state $(x,0)$, while assigning one job leads to state $(x-1,1)$.
Thus, if this difference is positive, then assigning one job to a slow server is better than assigning none.
The inequality above says that once assigning one job is optimal at~$x$, it remains optimal at $x+1$.
So \eqref{eq:multimodular-1} gives monotonicity of the assignment decision as the queue length increases.

Property \eqref{eq:multimodular-1}, when applied to $\Val$ at $y=1$, also gives the analogous monotonicity when one slow server is already busy:
\begin{equation}
\Val(x,1)-\Val(x-1,2)
\ge \Val(x-1,1)-\Val(x-2,2).
\end{equation}
Here $\Val(x-1,1)-\Val(x-2,2)$ compares assigning no additional job to the one idle slow server with assigning one additional job.
If it is positive, then assigning one additional job is better at $x-1$, and the inequality above says it is also better at $x$.
Therefore, overall, property \eqref{eq:multimodular-1} implies that as the queue grows, the incentive to use slow servers can only increase.

The property \eqref{eq:convex-y-fix-total} plays a different role.
Applying it to $\Val$ gives
\begin{equation}
\Val(x,0)-\Val(x-1,1)
\ge\Val(x-1,1)-\Val(x-2,2).
\end{equation}
The left hand side compares assigning no job with assigning one job to the slow servers, and the right hand side compares assigning one job with assigning two jobs to the slow servers.
This inequality makes it impossible to have
\begin{equation}
\Val(x,0)-\Val(x-1,1)<0
\quad\text{and}\quad
\Val(x-1,1)-\Val(x-2,2)\ge0.
\end{equation}
That is, it rules out the pattern in which assigning one job is worse than assigning none, but assigning two jobs is better than assigning one.

\begin{lemma}\label{lem:threshold}
    If the optimal value function $\Val\in \K$, then the optimal policy $\pi^*$ given in \eqref{eq:opt-policy} has the threshold structure in Theorem~\ref{thm:main}.
\end{lemma}
\begin{myproof}
Consider the value function $\Val \in \K$, and define
\begin{align}
\Dzo(x)
&= \Val(x,0)-\Val(x-1,1),
&& x\ge 1, \label{eq:def-d01} \\
\Dtw(x)
&= \Val(x-1,1)-\Val(x-2,2),
&& x\ge 2. \label{eq:def-d12}
\end{align}
Then the policy $\piopt$ in \eqref{eq:opt-policy} can be expressed using $\Dzo$ and $\Dtw$ as follows.
\begin{itemize}
    \item If $x=0$ or $1$,
    \begin{align}
    \piopt(0,y)=y, \qquad \piopt(1,y)=y,\qquad y=0,1,2.
    \label{eq:policy-small-x}
    \end{align}
    \item If $x=2$,
    \begin{align}
    \piopt(x,0)
    &=
    \begin{cases}
    0, & \text{if } \Dzo(x)\le 0,\\
    1, & \text{if } \Dzo(x)>0,
    \end{cases} \label{eq:policy-x2-y0} \\
    \piopt(x,1)
    &=
    \begin{cases}
    1, & \text{if } \Dtw(x+1)\le 0,\\
    2, & \text{if } \Dtw(x+1)>0,
    \end{cases} \label{eq:policy-x2-y1}\\
    \piopt(x,2)
    &=2.
    \end{align}
    \item If $x\ge 3$,
    \begin{align}
    \piopt(x,0)
    &=
    \begin{cases}
    0, & \text{if } \Dzo(x)\le 0,\\
    1, & \text{if } \Dzo(x)>0 \text{ and } \Dtw(x)\le 0,\\
    2, & \text{if } \Dtw(x)>0,
    \end{cases} \label{eq:policy-large-x-y0} \\
    \piopt(x,1)
    &=
    \begin{cases}
    1, & \text{if } \Dtw(x+1)\le 0,\\
    2, & \text{if } \Dtw(x+1)>0,
    \end{cases} \label{eq:policy-large-x-y1}\\
    \piopt(x,2)
    &=2.
    \end{align}
\end{itemize}
We show below that the rules for $x\ge 2$ correspond to the threshold structure in Theorem~\ref{thm:main}.

We first show that $\Dzo(x)$ and $\Dtw(x)$ are nondecreasing in $x$. For all $x\ge 1$,
\begin{align}
\Dzo(x+1)-\Dzo(x)
&= \Val(x+1,0)-\Val(x,1)
   -\bigl(\Val(x,0)-\Val(x-1,1)\bigr) \label{eq:d01-difference-first-line} \\
&= \Val(x+1,0)+\Val(x-1,1)-\Val(x,0)-\Val(x,1) \label{eq:d01-difference-expanded} \\
&\ge 0, \label{eq:d01-monotone}
\end{align}
where the last inequality follows from \eqref{eq:multimodular-1} at $(x-1,0)$. Similarly, for all $x\ge 2$,
\begin{align}
\Dtw(x+1)-\Dtw(x)
&= \Val(x,1)-\Val(x-1,2)
   -\bigl(\Val(x-1,1)-\Val(x-2,2)\bigr) \label{eq:d12-difference-first-line} \\
&= \Val(x,1)+\Val(x-2,2)-\Val(x-1,1)-\Val(x-1,2) \label{eq:d12-difference-expanded} \\
&\ge 0, \label{eq:d12-monotone}
\end{align}
where the last inequality follows from \eqref{eq:multimodular-1} at $(x-2,1)$.

With the monotonicity of $\Dzo(x)$ and $\Dtw(x)$, we define
\begin{align}
\xl &= \min\{x\ge 1\colon  \Dzo(x)>0\}, \label{eq:def-xL} \\
\xh &= \min\{x\ge 2\colon \Dtw(x)>0\}, \label{eq:def-xH}
\end{align}
where the $\min$ is set to $\infty$ when the set is empty.
We now show that $\xl\le \xh$ using \eqref{eq:convex-y-fix-total}. To show $\xl\le \xh$, it suffices to show that
\begin{align}
\Dzo(x) \ge \Dtw(x), \qquad x\ge 2. \label{eq:d01-dominates-d12}
\end{align}
Indeed, if $\xh=\infty$, then $\xl\le\xh$;
if $\xh<\infty$,
then \eqref{eq:d01-dominates-d12} implies
\begin{align}
\Dzo(\xh)\ge \Dtw(\xh)>0, \label{eq:xH-implies-d01-positive}
\end{align}
and hence $\xl\le \xh$.

To prove \eqref{eq:d01-dominates-d12}, note that for all $x\ge 2$,
\begin{align}
\Dzo(x)-\Dtw(x)
&= \Val(x,0)-\Val(x-1,1)
   -\bigl(\Val(x-1,1)-\Val(x-2,2)\bigr) \label{eq:d01-minus-d12-first-line} \\
&= \Val(x,0)+\Val(x-2,2)-2\Val(x-1,1) \label{eq:d01-minus-d12-expanded} \\
&\ge 0, \label{eq:d01-minus-d12-nonnegative}
\end{align}
where the last inequality follows from \eqref{eq:convex-y-fix-total} at $x-2$.

Now, with $1\le \xl\le \xh \le \infty$ defined, let
\begin{align}
\ql &= \xl-1, \label{eq:def-qL} \\
\qh &= \xh-1, \label{eq:def-qH}
\end{align}
where we use the convention $\infty-1=\infty$.
Then $0\le \ql\le \qh \le \infty$.
We have the following equivalence relationships. For all $x\ge 2$, with $q=x-1$,
\begin{align}
\Dzo(x)\le 0 \quad
&\Leftrightarrow \quad x<\xl \quad
\Leftrightarrow \quad q<\ql, \label{eq:d01-threshold-equivalence} \\
\Dtw(x+1)\le 0 \quad
&\Leftrightarrow \quad x<\xh-1 \quad
\Leftrightarrow \quad q<\qh-1. \label{eq:d12-xplus1-threshold-equivalence}
\end{align}
For all $x\ge 3$, with $q=x-1$,
\begin{align}
\Dtw(x)\le 0 \quad
&\Leftrightarrow \quad x<\xh \quad
\Leftrightarrow \quad q<\qh. \label{eq:d12-threshold-equivalence}
\end{align}
Therefore, the optimal policy $\piopt$ can equivalently be written as follows. 
\begin{itemize}
\item If $x=0$ or $1$,
\begin{align}
\piopt(0,y)=y, \qquad \piopt(1,y)=y \qquad y=0,1,2. \label{eq:threshold-policy-small-x}
\end{align}
\item If $x=2$ and $q=x-1$,
\begin{align}
\piopt(x,0)
&=
\begin{cases}
0, & \text{if } q<\ql,\\
1, & \text{if } q\ge \ql,
\end{cases} \label{eq:threshold-policy-x2-y0} \\
\piopt(x,1)
&=
\begin{cases}
1, & \text{if } q<\qh-1,\\
2, & \text{if } q\ge \qh-1,
\end{cases} \label{eq:threshold-policy-x2-y1}\\
\piopt(x,2)&=2.
\end{align}
\item If $x\ge 3$ and $q=x-1$,
\begin{align}
\piopt(x,0)
&=
\begin{cases}
0, & \text{if } q<\ql,\\
1, & \text{if } \ql\le q<\qh,\\
2, & \text{if } q\ge \qh,
\end{cases} \label{eq:threshold-policy-large-x-y0} \\
\piopt(x,1)
&=
\begin{cases}
1, & \text{if } q<\qh-1,\\
2, & \text{if } q\ge \qh-1,
\end{cases} \label{eq:threshold-policy-large-x-y1}\\
\piopt(x,2)&=2.
\end{align}
\end{itemize}
This is the threshold structure in Theorem~\ref{thm:main}.
\ifarxiv
 
\else
\Halmos
\fi
\end{myproof}

\section{Proof for the Discounted-Cost Problem}\label{sec:discounted}

In this section, we prove the threshold structure in Theorem~\ref{thm:main} for the slow servers in the discounted-cost setting through a series of lemmas.

We perform the value iteration in \eqref{eq:vi-modified} starting from the zero function, which is a function in the cone $\K$.
With the construction of the cone $\K$ and Lemma~\ref{lem:threshold}, it suffices to show that the value iteration in \eqref{eq:vi-modified} preserves the cone $\K$ and that its limit $\widetilde{V}_{\alpha}^{(\infty)}$ coincides with the optimal value function $\Val$.
Below we first state Lemmas~\ref{lem:qclosure}--\ref{lem:vi-consistency} that establish these results, and we then prove these lemmas in Sections~\ref{sec:proof-lem-qclosure}--\ref{sec:proof-lem-vi-consistency}.
We have also verified Lemmas~\ref{lem:qclosure}--\ref{lem:vi-consistency} in Lean 4, 
\ifarxiv
with a GitHub repository published at \citet{GitLeanOneFastTwoSlow_26}.
\else
with an anonymized GitHub repository published at \citet{GitLeanOneFastTwoSlow_26_anonymous}.
\fi

Recall the value iteration in \eqref{eq:vi-modified} below
\begin{equation}
    \widetilde{V}_{\alpha}^{(n+1)}=\Bellmanmodified \widetilde{V}_{\alpha}^{(n)},
\end{equation}
where $\Bellmanmodified$ is defined as
\begin{equation}
    (\Bellmanmodified f)(x, y) = c(x,y) + \frac{\rho_{\alpha}}{\gamma} (\Pop \Qop f)(x,y),\qquad x\in\Zp,\;y\in\{0,1,2\}.
\end{equation}
To prove that this value iteration preserves $\K$, i.e., if $f\in\K$, then $\Bellmanmodified f\in \K$, we note the following.
It is not hard to verify that the cost function $c\in\K$.
Since $\K$ is closed under nonnegative linear combination, it suffices to prove that if $f\in\K$, then $\Pop \Qop f\in \K$.
This is established through Lemmas~\ref{lem:qclosure} and \ref{lem:pclosure} below.
The result that this value iteration converges to the optimal value function is established in Lemma~\ref{lem:vi-consistency}.

\begin{lemma}[Assignment closure]\label{lem:qclosure}
If $f\in\K$, then $\Qop f\in\K$.
\end{lemma}

\begin{lemma}[Event closure]\label{lem:pclosure}
If $f\in\K$, then $\Pop f\in\K$.
\end{lemma}

\begin{lemma}[Value function consistency]\label{lem:vi-consistency}
If $f\in \K$, then $\Bellmanmodified (\Bellmanmodified f)=\Bellman(\Bellmanmodified f)$.
Consequently, $\widetilde{V}_{\alpha}^{(\infty)}=\Val$.
\end{lemma}

We remark that the cone $\K$ is not closed under the original job assignment operator, $\Qtrue$.
To see this, consider the function $f(x,y)=x$.
We can easily verify that $f\in\K$.
However, note that 
\begin{equation}
    (\Qtrue f)(0,0)=0,\quad (\Qtrue f)(1,0)=1,\quad (\Qtrue f)(2,0)=1.
\end{equation}
Then
\begin{equation}
    (\Qtrue f)(2,0)+(\Qtrue f)(0,0)-2(\Qtrue f)(1,0)=-1<0,
\end{equation}
which violates the convexity property in $x$ in \eqref{eq:convex-x}.
However, by Lemma~\ref{lem:vi-consistency}, for any $f\in\K$, if we apply $\Bellmanmodified$ first to get $\Bellmanmodified f$, when $\Bellman(\Bellmanmodified f)=\Bellmanmodified(\Bellmanmodified f)\in\K$.

\subsection{Proof of Lemma~\ref{lem:qclosure}}
\label{sec:proof-lem-qclosure}
\begin{myproof}
Consider any $f\in \K$, and we write $\bar f=\Qop f$ for convenience in this proof.  
Similar to the proof of Lemma~\ref{lem:threshold}, we first define
\begin{align}
    D_{01}(x)&=f(x,0)-f(x-1,1),
    && x\ge1, \label{eq:def-D01}\\
    D_{12}(x)&=f(x-1,1)-f(x-2,2),
    &&x\ge2.   \label{eq:def-D12}
\end{align}
Also following arguments similar to those in the proof of Lemma~\ref{lem:threshold}, we can show that $D_{01}(x)$ and $D_{12}(x)$ are nondecreasing in $x$ and $D_{01}(x)\ge D_{12}(x)$ for all $x\ge 2$.
In particular, \eqref{eq:multimodular-1} is equivalent to
\begin{align}
D_{01}(x+1)-D_{01}(x)
&=f(x+1,0)+f(x-1,1)-f(x,0)-f(x,1)\ge0,\qquad x\ge1, \label{eq:D01-monotone}\\
D_{12}(x+1)-D_{12}(x)
&=f(x,1)+f(x-2,2)-f(x-1,1)-f(x-1,2)\ge0,\qquad x\ge2, \label{eq:D12-monotone}
\end{align}
and \eqref{eq:convex-y-fix-total} is equivalent to
\begin{equation}
D_{01}(x)-D_{12}(x)
=f(x,0)+f(x-2,2)-2f(x-1,1)\ge0,
\qquad x\ge2.                                    \label{eq:gap-order}
\end{equation}

We then consider the number of jobs at the slow servers that achieves the minimum in $\Qop$:
\begin{equation}
    a(x,y)=\min\;\operatorname*{arg\;min}\limits_{y\le k\le\min\{2,x+y\}} f(x+y-k,k).       \label{eq:assignment-selector-def}
\end{equation}
Note that here since we use the modified job assignment operator $\Qop$, the upper constraint for $k$ is relaxed to $k\le\min\{2,x+y\}$.
It is easy to see that $a(1,0)\le a(0,1)$ and $a(x,0)\le a(x-1,1)\le a(x-2,2)$ for all $x\ge 2$ due to the optimization range.
Also, by definition,
\begin{equation}
    \bar f(x,y)=f(x+y-a(x,y),a(x,y)).                             \label{eq:barf-selector-representation}
\end{equation}

Using $D_{01}(x)$ and $D_{12}(x)$, we can write $a(x,y)$ as follows.
For $y=0$,
\begin{equation}
    a(0,0)=0,
\end{equation}
\begin{equation}
    a(1,0)=
    \begin{cases}
    0, & \text{if } D_{01}(1)\le0,\\
    1, & \text{if } D_{01}(1)>0,
    \end{cases}                                           \label{eq:assignment-selector-boundary10}
\end{equation}
and for $x\ge2$,
\begin{equation}
        a(x,0)=
        \begin{cases}
        0, & \text{if } D_{01}(x)\le0,\\
        1, & \text{if } D_{01}(x)>0\text{ and }D_{12}(x)\le0,\\
        2, & \text{if } D_{12}(x)>0.
        \end{cases}                                           \label{eq:assignment-selector-y0}
\end{equation}
Similarly, for $y=1$,
\begin{equation}
    a(0,1)=1,
\end{equation}
and for $x\ge1$,
\begin{equation}
        a(x,1)=
        \begin{cases}
        1, & \text{if } D_{12}(x+1)\le0,\\
        2, & \text{if } D_{12}(x+1)>0.
        \end{cases}                                           \label{eq:assignment-selector-y1}
\end{equation}
Also note that for $y=2$, by definition,
\begin{equation}
    a(x,2)=2.
\end{equation}
The properties of $D_{01}(x)$ and $D_{12}(x)$ in \eqref{eq:D01-monotone}--\eqref{eq:gap-order} imply that $a(x,0)\in\{0,1,2\}$ and $a(x,1)\in\{1,2\}$ are nondecreasing in $x$, and that for all $x\ge 2$,
\begin{equation}\label{eq:selectors-D}
    \begin{cases}
        a(x,0)=0,\; a(x-1,1)=1,\quad &\text{if }D_{01}(x)\le 0,\\
        a(x,0)=a(x-1,1)\ge 1,\quad&\text{if }D_{01}(x) > 0.
    \end{cases}
\end{equation}
We will also use the following \emph{selection condition} repeatedly: if $a(x,y)=k^*$, then
\begin{equation}
    f(x+y-k^*,k^*)\le f(x+y-k,k)
    \qquad \forall k\in\{y,\ldots,\min\{2,x+y\}\}.      \label{eq:assignment-selection-condition}
\end{equation}

We are now ready to prove that $\bar f$ satisfies \eqref{eq:mono-x}--\eqref{eq:convex-y-fix-total}.
In all the arguments below, the boundary cases for $x$ are covered, and all the inequalities hold as long as the input to $f$ is in $\Zp\times\{0,1,2\}$.

\paragraph{Proving $\bar f$ satisfies \eqref{eq:multimodular-1} and \eqref{eq:convex-y-fix-total}.}
Define
\begin{align}
    \bar D_{01}(x)&=\bar f(x,0)-\bar f(x-1,1), && x\ge1,
    \label{eq:def-bar-D01}\\
     \bar D_{12}(x)&=\bar f(x-1,1)-\bar f(x-2,2),&& x\ge2.
    \label{eq:def-bar-D12}
\end{align}
Then proving that $\bar f$ satisfies \eqref{eq:multimodular-1} and \eqref{eq:convex-y-fix-total} is equivalent to proving that $\bar D_{01}(x)$ and $\bar D_{12}(x)$ are nondecreasing in $x$ and $\bar D_{01}(x)\ge \bar D_{12}(x)$ for all $x\ge 2$.

By the property in \eqref{eq:selectors-D}, it is not hard to verify that
\begin{align}
    \bar D_{01}(x)&=\min\{D_{01}(x),0\},&& x\ge1,       \label{eq:bar-D01-min}\\
    \bar D_{12}(x)&=\min\{D_{12}(x),0\},&& x\ge2.       \label{eq:bar-D12-min}
\end{align}
Therefore, since $D_{01}(x)$ and $D_{12}(x)$ are nondecreasing in $x$ and $D_{01}(x)\ge D_{12}(x)$ for all $x\ge 2$, we know that $\bar D_{01}(x)$ and $\bar D_{12}(x)$ are nondecreasing in $x$ and $\bar D_{01}(x)\ge \bar D_{12}(x)$ for all $x\ge 2$.

\paragraph{Proving $\bar f$ satisfies \eqref{eq:mono-x}.}
The desired inequality is
\begin{equation}
    \bar f(x+1,y)-\bar f(x,y)\ge0,\qquad y=0,1,2,\;x\ge 0.                       \label{eq:bar-K1-desired}
\end{equation}
The minimizers involved are
\begin{equation}
    a_0=a(x,y),\qquad a_1=a(x+1,y), \label{eq:K1-selectors}
\end{equation}
which are ordered as $a_0\le a_1$ by the monotonicity of the minimizers.
Then \eqref{eq:bar-K1-desired} is equivalent to
\begin{equation}
    f(x+1+y-a_1,a_1)-f(x+y-a_0,a_0)\ge0,\qquad y=0,1,2,\;x\ge 0. \label{eq:bar-K1-desired-equiv}
\end{equation}
\begin{itemize}
\item If $a_0=a_1$, then \eqref{eq:bar-K1-desired-equiv} directly follows from \eqref{eq:mono-x} for $f$.
\item If $(a_0,a_1)=(0,1)$, then necessarily $y=0$, and \eqref{eq:bar-K1-desired-equiv} becomes
\begin{equation}
    f(x,1)-f(x,0)\ge0,                                  \label{eq:K1-pattern-01}
\end{equation}
which follows from \eqref{eq:mono-y} for $f$.  
\item If $(a_0,a_1)=(1,2)$, then \eqref{eq:bar-K1-desired-equiv} becomes: if $y=0$,
\begin{equation}
        f(x-1,2)-f(x-1,1)\ge0,                                \label{eq:K1-pattern-12-y0}
\end{equation}
or if $y=1$
\begin{equation}
        f(x,2)-f(x,1)\ge0.                                    \label{eq:K1-pattern-12-y1}
\end{equation}
Both follow from \eqref{eq:mono-y} for $f$.  
\item If $(a_0,a_1)=(0,2)$, then necessarily $y=0$, and \eqref{eq:bar-K1-desired-equiv} becomes
\begin{equation}
    f(x-1,2)-f(x,0)
    =[f(x-1,2)-f(x-1,1)]+[f(x-1,1)-f(x,0)]\ge0,          \label{eq:K1-skip-02}
\end{equation}
where the first bracket follows from \eqref{eq:mono-y}, and the second is due to the selection condition associated with $a_0=a(x,0)=0$.
\end{itemize}
This completes the proof that $\bar f$ satisfies \eqref{eq:mono-x}.

\paragraph{Proving $\bar f$ satisfies \eqref{eq:multimodular-2}.}
The desired inequality is
\begin{equation}
    \bar f(x,y)+\bar f(x+1,y+1)-\bar f(x+1,y)-\bar f(x,y+1)\ge0,\qquad y=0,1,\;x\ge 0.     
    \label{eq:bar-K3-desired}
\end{equation}
The minimizers involved are
\begin{equation}
        a_{00}=a(x,y),\quad a_{10}=a(x+1,y),\quad
        a_{01}=a(x,y+1),\quad a_{11}=a(x+1,y+1).             \label{eq:K3-selectors}
\end{equation}
They are ordered as $a_{00}\le a_{10}$ and $a_{01}\le a_{11}$ due to monotonicity, and $a_{10}\le a_{01}$ by definition.
Then \eqref{eq:bar-K3-desired} is equivalent to
\begin{equation}\label{eq:Q-closure-K3-expanded}
\begin{aligned}
    &f(x+y-a_{00},a_{00})+ f(x+y+2-a_{11},a_{11})\\
    &- f(x+y+1-a_{10},a_{10})- f(x+y+1-a_{01},a_{01})\ge0,\qquad y=0,1,\;x\ge 0.    
\end{aligned}
\end{equation}

We first consider the case $y=0$.
The possible values for $(a_{00},a_{10},a_{01},a_{11})$ are listed below, which are grouped based on the value of $(a_{01},a_{11})$:
\begin{align}
&(0,0,1,1),\ (0,1,1,1),\ (1,1,1,1),\\
&(0,0,1,2),\ (0,1,1,2),\ (1,1,1,2),\\
&(0,0,2,2),\ (0,1,2,2),\ (0,2,2,2),\ (1,1,2,2),\ (1,2,2,2),\ (2,2,2,2).
\label{eq:Q-closure-K3-y0-patterns}
\end{align}
We check these cases one by one.
\begin{itemize}
\item 
If $(a_{00},a_{10},a_{01},a_{11})=(0,0,1,1)$, then
\eqref{eq:Q-closure-K3-expanded} becomes
\begin{equation}
f(x,0)+f(x+1,1)-f(x+1,0)-f(x,1)\ge0,
\label{eq:Q-closure-K3-y0-0011}
\end{equation}
which follows from \eqref{eq:multimodular-2} for $f$ with $y=0$.

\item If $(a_{00},a_{10},a_{01},a_{11})=(0,1,1,1)$, then
\eqref{eq:Q-closure-K3-expanded} becomes
\begin{align}
& f(x,0)+f(x+1,1)-2f(x,1) \notag\\
&\quad =
\bigl[f(x,0)+f(x+1,1)-f(x+1,0)-f(x,1)\bigr]
+\bigl[f(x+1,0)-f(x,1)\bigr]\ge0,
\label{eq:Q-closure-K3-y0-0111}
\end{align}
where the first bracket is nonnegative by \eqref{eq:multimodular-2} for $f$ with $y=0$, and the second bracket is nonnegative
by the selection condition associated with $a_{10}=a(x+1,0)=1$.

\item If $(a_{00},a_{10},a_{01},a_{11})=(1,1,1,1)$, then
\eqref{eq:Q-closure-K3-expanded} becomes
\begin{equation}
f(x-1,1)+f(x+1,1)-2f(x,1)\ge0,
\label{eq:Q-closure-K3-y0-1111}
\end{equation}
which follows from the convexity property of $f(x,y)$ in $x$ in \eqref{eq:convex-x}.

\item 
If $(a_{00},a_{10},a_{01},a_{11})=(0,0,1,2)$, then
\eqref{eq:Q-closure-K3-expanded} becomes
\begin{align}
& f(x,0)+f(x,2)-f(x+1,0)-f(x,1) \notag\\
&\quad =
\bigl[f(x,0)+f(x,2)-2f(x,1)\bigr]
+\bigl[f(x,1)-f(x+1,0)\bigr]\ge0.
\label{eq:Q-closure-K3-y0-0012}
\end{align}
The first bracket is nonnegative by \eqref{eq:convex-y-fix-x} for $f$, and the second bracket is nonnegative by the
selection condition associated with $a_{10}=a(x+1,0)=0$.

\item  If $(a_{00},a_{10},a_{01},a_{11})=(0,1,1,2)$, then
\eqref{eq:Q-closure-K3-expanded} becomes
\begin{equation}
f(x,0)+f(x,2)-2f(x,1)\ge0,
\label{eq:Q-closure-K3-y0-0112}
\end{equation}
which is \eqref{eq:convex-y-fix-x} for $f$.

\item  If $(a_{00},a_{10},a_{01},a_{11})=(1,1,1,2)$, then
\eqref{eq:Q-closure-K3-expanded} becomes
\begin{align}
& f(x-1,1)+f(x,2)-2f(x,1) \notag\\
&\quad =
\bigl[f(x-1,1)+f(x,2)-f(x,1)-f(x-1,2)\bigr]
+\bigl[f(x-1,2)-f(x,1)\bigr]\ge0.
\label{eq:Q-closure-K3-y0-1112}
\end{align}
The first bracket is nonnegative by \eqref{eq:multimodular-2} for $f$ with $y=1$, evaluated at $x-1$, and the second
bracket is nonnegative by the selection condition associated with $a_{01}=a(x,1)=1$.

\item If $(a_{00},a_{10},a_{01},a_{11})=(0,0,2,2)$, then
\eqref{eq:Q-closure-K3-expanded} becomes
\begin{align}
& f(x,0)+f(x,2)-f(x+1,0)-f(x-1,2) \notag\\
&\quad =
\bigl[f(x,0)+f(x,2)-2f(x,1)\bigr]
+\bigl[f(x,1)-f(x+1,0)\bigr]
+\bigl[f(x,1)-f(x-1,2)\bigr]\ge0.
\label{eq:Q-closure-K3-y0-0022}
\end{align}
The first bracket is nonnegative by \eqref{eq:convex-y-fix-x}, the second by the selection condition associated with
$a_{10}=a(x+1,0)=0$, and the third by the selection condition associated with $a_{01}=a(x,1)=2$.

\item If $(a_{00},a_{10},a_{01},a_{11})=(0,1,2,2)$, then
\eqref{eq:Q-closure-K3-expanded} becomes
\begin{align}
& f(x,0)+f(x,2)-f(x,1)-f(x-1,2) \notag\\
&\quad =
\bigl[f(x,0)+f(x,2)-2f(x,1)\bigr]
+\bigl[f(x,1)-f(x-1,2)\bigr]\ge0.
\label{eq:Q-closure-K3-y0-0122}
\end{align}
The first bracket is nonnegative by \eqref{eq:convex-y-fix-x}, and the second is nonnegative by the selection condition
associated with $a_{01}=a(x,1)=2$.

\item If $(a_{00},a_{10},a_{01},a_{11})=(0,2,2,2)$, then
\eqref{eq:Q-closure-K3-expanded} becomes
\begin{align}
& f(x,0)+f(x,2)-2f(x-1,2) \notag\\
&\quad =
\bigl[f(x,0)+f(x,2)-2f(x,1)\bigr]
+2\bigl[f(x,1)-f(x-1,2)\bigr]\ge0.
\label{eq:Q-closure-K3-y0-0222}
\end{align}
The first bracket is nonnegative by \eqref{eq:convex-y-fix-x}, and the second bracket is nonnegative by the selection
condition associated with $a_{01}=a(x,1)=2$.

\item If $(a_{00},a_{10},a_{01},a_{11})=(1,1,2,2)$, then
\eqref{eq:Q-closure-K3-expanded} becomes
\begin{equation}
f(x-1,1)+f(x,2)-f(x,1)-f(x-1,2)\ge0,
\label{eq:Q-closure-K3-y0-1122}
\end{equation}
which is \eqref{eq:multimodular-2} for $f$ with $y=1$, evaluated at $x-1$.

\item If $(a_{00},a_{10},a_{01},a_{11})=(1,2,2,2)$, then
\eqref{eq:Q-closure-K3-expanded} becomes
\begin{align}
& f(x-1,1)+f(x,2)-2f(x-1,2) \notag\\
&\quad =
\bigl[f(x-1,1)+f(x,2)-f(x,1)-f(x-1,2)\bigr]
+\bigl[f(x,1)-f(x-1,2)\bigr]\ge0.
\label{eq:Q-closure-K3-y0-1222}
\end{align}
The first bracket is nonnegative by \eqref{eq:multimodular-2} for $f$ with $y=1$, evaluated at $x-1$, and the second
bracket is nonnegative by the selection condition associated with $a_{01}=a(x,1)=2$.

\item Finally, if $(a_{00},a_{10},a_{01},a_{11})=(2,2,2,2)$, then
\eqref{eq:Q-closure-K3-expanded} becomes
\begin{equation}
f(x-2,2)+f(x,2)-2f(x-1,2)\ge0,
\label{eq:Q-closure-K3-y0-2222}
\end{equation}
which follows from the convexity property of $f(x,y)$ in $x$ in \eqref{eq:convex-x}.
\end{itemize}

We now consider the case $y=1$.  Then
\begin{equation}
a_{01}=a(x,2)=2,
\qquad
a_{11}=a(x+1,2)=2,
\label{eq:Q-closure-K3-y1-fixed}
\end{equation}
and the possible pairs $(a_{00},a_{10})$ are
\begin{equation}
(1,1),\qquad (1,2),\qquad (2,2).
\label{eq:Q-closure-K3-y1-pairs}
\end{equation}
\begin{itemize}
\item If $(a_{00},a_{10})=(1,1)$, then \eqref{eq:Q-closure-K3-expanded} becomes
\begin{equation}
f(x,1)+f(x+1,2)-f(x+1,1)-f(x,2)\ge0,
\label{eq:Q-closure-K3-y1-11}
\end{equation}
which is \eqref{eq:multimodular-2} for $f$ with $y=1$.

\item 
If $(a_{00},a_{10})=(1,2)$, then \eqref{eq:Q-closure-K3-expanded} becomes
\begin{align}
& f(x,1)+f(x+1,2)-2f(x,2) \notag\\
&\quad =
\bigl[f(x,1)+f(x+1,2)-f(x+1,1)-f(x,2)\bigr]
+\bigl[f(x+1,1)-f(x,2)\bigr]\ge0.
\label{eq:Q-closure-K3-y1-12}
\end{align}
The first bracket is nonnegative by \eqref{eq:multimodular-2} for $f$ with $y=1$, and the second bracket is nonnegative
by the selection condition associated with $a_{10}=a(x+1,1)=2$.

\item If $(a_{00},a_{10})=(2,2)$, then \eqref{eq:Q-closure-K3-expanded} becomes
\begin{equation}
f(x-1,2)+f(x+1,2)-2f(x,2)\ge0,
\label{eq:Q-closure-K3-y1-22}
\end{equation}
which follows from the convexity property of $f(x,y)$ in $x$ in \eqref{eq:convex-x}.
\end{itemize}

This completes the proof that $\bar f$ satisfies \eqref{eq:multimodular-2}.

\paragraph{Proving $\bar f$ satisfies \eqref{eq:mono-y}.}
The desired inequality is
\begin{equation}
    \bar f(x,y+1)-\bar f(x,y)\ge0,
    \qquad y=0,1,\; x\ge 0.                                      \label{eq:bar-K4-desired}
\end{equation}
The minimizers involved are
\begin{equation}
    a_0=a(x,y),\qquad a_1=a(x,y+1),
\end{equation}
which are ordered as $a_0\le a_1$ since $a_0\le a(x+1,y)\le a_1$.
Then \eqref{eq:bar-K4-desired} is equivalent to
\begin{equation}
    f(x+y+1-a_1,a_1)-f(x+y-a_0,a_0)\ge0,
    \qquad y=0,1,\; x\ge 0.                                      \label{eq:bar-K4-desired-equiv}
\end{equation}
\begin{itemize}
\item For $y=0$, the possible values for $(a_0,a_1)$ are $(0,1),(0,2),(1,1),(1,2),(2,2)$.  
For $(a_0,a_1)=(0,1)$, the inequality \eqref{eq:bar-K4-desired-equiv} is $f(x,1)-f(x,0)\ge0$, which follows from \eqref{eq:mono-y}.
For $(a_0,a_1)=(0,2)$, the inequality \eqref{eq:bar-K4-desired-equiv} is $f(x-1,2)-f(x,0)\ge 0$, which has been proven in \eqref{eq:K1-skip-02}.  
For $(a_0,a_1)=(1,1)$, the inequality \eqref{eq:bar-K4-desired-equiv} is $f(x,1)-f(x-1,1)\ge0$, which follows from \eqref{eq:mono-x}.  
For $(a_0,a_1)=(1,2)$, the inequality \eqref{eq:bar-K4-desired-equiv} is $f(x-1,2)-f(x-1,1)\ge0$, which follows from \eqref{eq:mono-y}.  
For $(a_0,a_1)=(2,2)$, the inequality \eqref{eq:bar-K4-desired-equiv} is $f(x-1,2)-f(x-2,2)\ge0$, which follows from \eqref{eq:mono-x}.

\item For $y=1$, the possible values for $(a_0,a_1)$ are $(1,2)$ and $(2,2)$, which have been covered when checking the case where $y=0$.
\end{itemize}
This completes the proof that $\bar f$ satisfies \eqref{eq:mono-y}.

\paragraph{Proving $\bar f$ satisfies \eqref{eq:convex-y-fix-x}.}
The desired inequality is
\begin{equation}
    \bar f(x,0)+\bar f(x,2)-2\bar f(x,1)\ge0,\qquad x\ge 0.           \label{eq:bar-K5-desired}
\end{equation}
The minimizers involved are
\begin{equation}
    a_0=a(x,0),\qquad a_1=a(x,1),\qquad a_2=a(x,2)=2,         \label{eq:K5-selectors}
\end{equation}
which are ordered as $a_0\le a_1\le a_2$ since $a_0=a(x,0)\le a(x+1,0)\le a(x,1)=a_1\le a_2=2$.
Then \eqref{eq:bar-K5-desired} is equivalent to
\begin{equation}
    f(x-a_0,a_0)+f(x,2)-2 f(x+1-a_1,a_1)\ge0,\qquad x\ge 0.  \label{eq:bar-K5-desired-equiv}
\end{equation}
There are five possibilities for $(a_0,a_1,a_2)$, which are $(0,1,2),(0,2,2),(1,1,2),(1,2,2),(2,2,2)$.
\begin{itemize}
\item If $(a_0,a_1,a_2)=(0,1,2)$, then \eqref{eq:bar-K5-desired-equiv} is $f(x,0)+f(x,2)-2 f(x,1)\ge0$, which is \eqref{eq:convex-y-fix-x} for $f$.

\item If $(a_0,a_1,a_2)=(0,2,2)$, then \eqref{eq:bar-K5-desired-equiv} becomes
\begin{align}
&f(x,0)+f(x,2)-2f(x-1,2)\notag\\
&\quad =[f(x,0)+f(x,2)-2f(x,1)]
      +2[f(x,1)-f(x-1,2)]\ge0,         \label{eq:K5-pattern-022}
\end{align}
where the first bracket is nonnegative due to \eqref{eq:convex-y-fix-x}, and the second bracket is nonnegative due to the selection condition $a_1=a(x,1)=2$.

\item If $(a_0,a_1,a_2)=(1,1,2)$, then \eqref{eq:bar-K5-desired-equiv} becomes
\begin{align}
&f(x-1,1)+f(x,2)-2f(x,1)\notag\\
&\quad =[f(x-1,1)+f(x,2)-f(x,1)-f(x-1,2)]
      +[f(x-1,2)-f(x,1)]\ge0,                            \label{eq:K5-pattern-112}
\end{align}
where the first bracket is nonnegative due to \eqref{eq:multimodular-2}, and the second bracket is nonnegative due to the selection condition $a_1=a(x,1)=1$.

\item If $(a_0,a_1,a_2)=(1,2,2)$, then \eqref{eq:bar-K5-desired-equiv} becomes
\begin{align}
&f(x-1,1)+f(x,2)-2f(x-1,2)\notag\\
&\quad =[f(x-1,1)+f(x,2)-f(x,1)-f(x-1,2)]
+[f(x,1)-f(x-1,2)]\ge0,                \label{eq:K5-pattern-122}
\end{align}
where the first bracket is nonnegative due to \eqref{eq:multimodular-2}, and the second bracket is nonnegative due to the selection condition $a_1=a(x,1)=2$.

\item Finally, if $(a_0,a_1,a_2)=(2,2,2)$, then \eqref{eq:bar-K5-desired-equiv} becomes
\begin{equation}
    f(x-2,2)+f(x,2)-2f(x-1,2)\ge 0,
\end{equation}
which follows from the convexity property of $f(x,y)$ in $x$ in \eqref{eq:convex-x}.
\end{itemize}
This completes the proof that $\bar f$ satisfies \eqref{eq:convex-y-fix-x}.

Now we have proved that $\bar f =\Qop f$ satisfies \eqref{eq:mono-x}--\eqref{eq:convex-y-fix-total}, and thus $\bar f \in \K$, which completes the proof of Lemma~\ref{lem:qclosure}.
\ifarxiv

\else
\Halmos
\fi
\end{myproof}

\subsection{Proof of Lemma~\ref{lem:pclosure}}
\label{sec:proof-lem-pclosure}

\begin{myproof}
Consider any $f\in \K$.
We first decompose the overall event operator $\Pop$ into three separate event operators that correspond to the events of a job arrival, a job departure from the fast server, and a job departure from the slow servers, respectively:
\begin{itemize}
\item \emph{Arrival operator $\Aop$}:
\begin{equation}
(\Aop f)(x,y)=f(x+1,y),
\qquad y=0,1,2,\ x\ge 0,
\label{eq:event-arrival-operator}
\end{equation}

\item \emph{Fast departure operator $\DFop$}:
\begin{equation}
\begin{aligned}
(\DFop f)(x,y)&=f((x-1)^+,y)\\
&=
\begin{cases}
    f(0,y), \quad&\text{if }x=0,\\
    f(x-1,y),\quad&\text{if }x\ge 1,
\end{cases}
\qquad y=0,1,2,
\end{aligned}
\label{eq:event-fast-operator}
\end{equation}

\item \emph{Slow departure operator $\DSop$}:
\begin{equation}\label{eq:event-slow-operator}
\begin{aligned}
(\DSop f)(x,0)&=2f(x,0),\qquad &&x\ge 0,\\
(\DSop f)(x,1)&=f(x,0)+f(x,1),\qquad &&x\ge 0,\\
(\DSop f)(x,2)&=2f(x,1),\qquad &&x\ge 0.
\end{aligned}
\end{equation}
\end{itemize}
Then
\begin{equation}
    \Pop f=\lambda\Aop f+\mu_1\DFop f
    +\mu_2\DSop f.
\label{eq:event-operator-decomposition}
\end{equation}
Since \(\mathcal K\) is closed under nonnegative linear combinations, it suffices to show that \(\Aop f \in \K,\DFop f\in\K,\DSop f\in \K\).

We next rewrite the cone properties \eqref{eq:mono-x}--\eqref{eq:convex-y-fix-total} into forms convenient for this proof.
For any function \(h:\mathbb Z_+\times\{0,1,2\}\to\mathbb R\), define
\begin{align}
L_1^h(x,y)
&=h(x+1,y)-h(x,y),
&&y=0,1,2,\ x\ge0,                                      \label{eq:event-L1}\\
L_2^h(x,y)
&=h(x+2,y)+h(x,y+1)-h(x+1,y)-h(x+1,y+1),
&&y=0,1,\ x\ge0,                                      \label{eq:event-L2}\\
L_3^h(x,y)
&=h(x,y)+h(x+1,y+1)-h(x+1,y)-h(x,y+1),
&&y=0,1,\ x\ge0,                                      \label{eq:event-L3}\\
L_4^h(x,y)
&=h(x,y+1)-h(x,y),
&&y=0,1,\ x\ge0,                                      \label{eq:event-L4}\\
L_5^h(x)
&=h(x,0)+h(x,2)-2h(x,1),
&&x\ge0,                                               \label{eq:event-L5}\\
L_6^h(x)
&=h(x+2,0)+h(x,2)-2h(x+1,1),
&&x\ge0.                                               \label{eq:event-L6}
\end{align}
Thus \(h\in\mathcal K\) if and only if $L_i^h\ge 0,i=1,2,3,4,5,6$, which correspond to \eqref{eq:mono-x}--\eqref{eq:convex-y-fix-total} for $h$.
We also use the convexity property in $x$ given in \eqref{eq:convex-x}, which can be written as
\begin{equation}
L_7^h(x,y)
:=h(x+2,y)+h(x,y)-2h(x+1,y)\ge0,
\qquad y=0,1,2,\ x\ge0.
\label{eq:event-L7}
\end{equation}
We have shown that this follows from \eqref{eq:multimodular-1} and \eqref{eq:multimodular-2} for $h$.

We now prove that $\K$ is closed under the arrival operator $\Aop$.
Let \(g=\Aop f\).  Then, for
\(i=1,\dots,6\),
\begin{equation}
        L_i^g(x,\cdot)=L_i^f(x+1,\cdot),
\label{eq:event-arrival-transport}
\end{equation}
with the obvious interpretation for \(L_5^f,L_5^g,L_6^f,L_6^g\), which do not have a \(y\)-argument.
Since $L_i^f(x,y)\ge 0$ for all $x\ge 0$, we have $L_i^f(x+1,\cdot)\ge 0$ for all $x\ge 0$.
Therefore, if \(f\in\mathcal K\), then \(g=\Aop f\in\mathcal K\).

We next prove that $\K$ is closed under the fast departure operator $\DFop$.
Let \(g=\DFop f\).
Using the definition of $\DFop$ in \eqref{eq:event-fast-operator} and noting that $g(0,y)=g(1,y)=f(0,y)$ for all $y=0,1,2$, we have the following.
\begin{itemize}
\item For \eqref{eq:mono-x},
\begin{equation}
L_1^g(x,y)=
\begin{cases}
0, & x=0,\\
L_1^f(x-1,y), & x\ge1,
\end{cases}
\qquad y=0,1,2.
\label{eq:event-fast-L1}
\end{equation}

\item For \eqref{eq:multimodular-1},
\begin{equation}
L_2^g(x,y)=
\begin{cases}
L_1^f(0,y), & x=0,\\
L_2^f(x-1,y), & x\ge1,
\end{cases}
\qquad y=0,1.
\label{eq:event-fast-L2}
\end{equation}

\item For \eqref{eq:multimodular-2},
\begin{equation}
L_3^g(x,y)=
\begin{cases}
0, & x=0,\\
L_3^f(x-1,y), & x\ge1,
\end{cases}
\qquad y=0,1.
\label{eq:event-fast-L3}
\end{equation}

\item For \eqref{eq:mono-y},
\begin{equation}
L_4^g(x,y)=L_4^f((x-1)^+,y),
\qquad x\ge 0,\ y=0,1.
\label{eq:event-fast-L4}
\end{equation}

\item For \eqref{eq:convex-y-fix-x},
\begin{equation}
L_5^g(x)=L_5^f((x-1)^+),
\qquad x\ge0.
\label{eq:event-fast-L5}
\end{equation}

\item For \eqref{eq:convex-y-fix-total},
\begin{equation}
L_6^g(x)=
\begin{cases}
L_1^f(0,0)+L_5^f(0), & x=0,\\
L_6^f(x-1), & x\ge1.
\end{cases}
\label{eq:event-fast-L6}
\end{equation}
\end{itemize}
Therefore, we can see that $L_i^g\ge 0$ for $i=1,2,\dots,6$, which implies that
\(g=\DFop f\in\mathcal K\).

Finally we prove that $\K$ is closed under the slow completion operator $\DSop$.
Let \(g=\DSop f\).
\begin{itemize}
\item For \eqref{eq:mono-x},
direct substitution gives
\begin{align}
L_1^g(x,0)
&=g(x+1,0)-g(x,0) \notag\\
&=2f(x+1,0)-2f(x,0) \notag\\
&=2L_1^f(x,0),                                           \label{eq:event-slow-L1-0}\\
L_1^g(x,1)
&=g(x+1,1)-g(x,1) \notag\\
&=\{f(x+1,0)+f(x+1,1)\}-\{f(x,0)+f(x,1)\} \notag\\
&=L_1^f(x,0)+L_1^f(x,1),                                 \label{eq:event-slow-L1-1}\\
L_1^g(x,2)
&=g(x+1,2)-g(x,2) \notag\\
&=2f(x+1,1)-2f(x,1) \notag\\
&=2L_1^f(x,1).                                           \label{eq:event-slow-L1-2}
\end{align}

\item For \eqref{eq:multimodular-1}, for all $x\ge 0$,
\begin{align}
L_2^g(x,0)
&=g(x+2,0)+g(x,1)-g(x+1,0)-g(x+1,1) \notag\\
&=2f(x+2,0)+\{f(x,0)+f(x,1)\}
  -2f(x+1,0)-\{f(x+1,0)+f(x+1,1)\} \notag\\
&=\bigl[f(x+2,0)+f(x,1)-f(x+1,0)-f(x+1,1)\bigr] \notag\\
&\quad
 +\bigl[f(x+2,0)+f(x,0)-2f(x+1,0)\bigr] \notag\\
&=L_2^f(x,0)+L_7^f(x,0),                                  \label{eq:event-slow-L2-0}
\\
L_2^g(x,1)
&=g(x+2,1)+g(x,2)-g(x+1,1)-g(x+1,2) \notag\\
&=\{f(x+2,0)+f(x+2,1)\}+2f(x,1) \notag\\
&\quad
 -\{f(x+1,0)+f(x+1,1)\}-2f(x+1,1) \notag\\
&=\bigl[f(x+2,0)+f(x,1)-f(x+1,0)-f(x+1,1)\bigr] \notag\\
&\quad
 +\bigl[f(x+2,1)+f(x,1)-2f(x+1,1)\bigr] \notag\\
&=L_2^f(x,0)+L_7^f(x,1).                                  \label{eq:event-slow-L2-1}
\end{align}

\item For \eqref{eq:multimodular-2},
\begin{align}
L_3^g(x,0)
&=g(x,0)+g(x+1,1)-g(x+1,0)-g(x,1) \notag\\
&=2f(x,0)+\{f(x+1,0)+f(x+1,1)\}
  -2f(x+1,0)-\{f(x,0)+f(x,1)\} \notag\\
&=f(x,0)+f(x+1,1)-f(x+1,0)-f(x,1) \notag\\
&=L_3^f(x,0),                                               \label{eq:event-slow-L3-0}\\
L_3^g(x,1)
&=g(x,1)+g(x+1,2)-g(x+1,1)-g(x,2) \notag\\
&=\{f(x,0)+f(x,1)\}+2f(x+1,1) \notag\\
&\quad
 -\{f(x+1,0)+f(x+1,1)\}-2f(x,1) \notag\\
&=f(x,0)+f(x+1,1)-f(x+1,0)-f(x,1) \notag\\
&=L_3^f(x,0).                                               \label{eq:event-slow-L3-1}
\end{align}

\item For \eqref{eq:mono-y},
\begin{align}
L_4^g(x,0)
&=g(x,1)-g(x,0) \notag\\
&=\{f(x,0)+f(x,1)\}-2f(x,0) \notag\\
&=f(x,1)-f(x,0) \notag\\
&=L_4^f(x,0),                                               \label{eq:event-slow-L4-0}\\
L_4^g(x,1)
&=g(x,2)-g(x,1) \notag\\
&=2f(x,1)-\{f(x,0)+f(x,1)\} \notag\\
&=f(x,1)-f(x,0) \notag\\
&=L_4^f(x,0).                                               \label{eq:event-slow-L4-1}
\end{align}

\item For \eqref{eq:convex-y-fix-x},
\begin{align}
L_5^g(x)
&=g(x,0)+g(x,2)-2g(x,1) \notag\\
&=2f(x,0)+2f(x,1)-2\{f(x,0)+f(x,1)\} \notag\\
&=0.                                                        \label{eq:event-slow-L5}
\end{align}

\item For \eqref{eq:convex-y-fix-total},
\begin{align}
L_6^g(x)
&=g(x+2,0)+g(x,2)-2g(x+1,1) \notag\\
&=2f(x+2,0)+2f(x,1)-2\{f(x+1,0)+f(x+1,1)\} \notag\\
&=2\bigl[f(x+2,0)+f(x,1)-f(x+1,0)-f(x+1,1)\bigr] \notag\\
&=2L_2^f(x,0).                                             \label{eq:event-slow-L6}
\end{align}
\end{itemize}
Again, we can see that $L_i^g\ge 0$ for $i=1,2,\dots,6$, which implies that
\(g=\DFop f\in\mathcal K\).

We have now shown that $\K$ is closed under each of
\(\Aop,\DFop,\DSop\),
which completes the proof of Lemma~\ref{lem:pclosure}.
\ifarxiv

\else
\Halmos
\fi
\end{myproof}

\subsection{Proof of Lemma~\ref{lem:vi-consistency}}
\label{sec:proof-lem-vi-consistency}
\begin{myproof}
We first prove that, for any $f\in\K$, after one iteration with $\Bellmanmodified$, the function $g = \Bellmanmodified f$ satisfies
\begin{equation}\label{eq:no-move-10-to-01}
    g(1,0) \le g(0,1),
\end{equation}
\begin{equation}\label{eq:no-move-11-to-02}
    g(1,1) \le g(0,2).
\end{equation}
That is, if $g$ is interpreted as a value function, then it is of lower cost to not move a job from the fast server to the slow servers.

Before giving the proof of \eqref{eq:no-move-10-to-01} and \eqref{eq:no-move-11-to-02}, we first show why they imply the desired result.  Note that \eqref{eq:no-move-10-to-01} and \eqref{eq:no-move-11-to-02} imply that
\begin{align}
    (\Qop g)(1,0)
    &= \min_{0\le k\le 1} g(1-k,k) \label{eq:tQ-at-10-min}\\
    &= g(1,0) \label{eq:tQ-at-10-g10}\\
    &= (\Qtrue g)(1,0), \label{eq:tQ-at-10-Q}
\end{align}
\begin{align}
    (\Qop g)(1,1)
    &= \min_{1\le k\le 2} g(2-k,k) \label{eq:tQ-at-11-min}\\
    &= g(1,1) \label{eq:tQ-at-11-g11}\\
    &= (\Qtrue g)(1,1), \label{eq:tQ-at-11-Q}
\end{align}
and
\begin{align}
    (\Qop g)(2,0)
    &= \min_{0\le k\le 2} g(2-k,k) \label{eq:tQ-at-20-min}\\
    &= \min_{0\le k\le 1} g(2-k,k) \label{eq:tQ-at-20-reduced}\\
    &= (\Qtrue g)(2,0). \label{eq:tQ-at-20-Q}
\end{align}
Since we also know that by definition,
$(\Qop g)(x,y) = (\Qtrue g)(x,y)$ for all $(x,y)$ that are not $(1,0)$, $(1,1)$, and $(2,0)$, we have $\Qop g = \Qtrue g$.
Thus, $\Bellmanmodified g = \Bellman g$, i.e.,
\begin{equation}\label{eq:tTg-equals-Tg}
    \Bellmanmodified (\Bellmanmodified f) = \Bellman (\Bellmanmodified f).
\end{equation}
Consequently, when we perform the value iteration with $\Bellmanmodified$ given in \eqref{eq:vi-modified} starting from $\widetilde V_{\alpha}^{(0)}\in\K$, we have
\begin{align}
    \widetilde V_{\alpha}^{(\infty)}
    &= \lim_{n\to\infty} \widetilde V_{\alpha}^{(n+1)} \label{eq:tilde-limit-shift}\\
    &= \lim_{n\to\infty} \Bellmanmodified^{n+1}\widetilde V_{\alpha}^{(0)} \label{eq:tilde-iterate}\\
    &= \lim_{n\to\infty} \Bellman^n\bigl(\Bellmanmodified \widetilde V_{\alpha}^{(0)}\bigr) \label{eq:tilde-to-original-iterate}\\
    &= V_{\alpha}^{(\infty)}, \label{eq:tilde-limit-equals-original}
\end{align}
as desired.

It remains to prove \eqref{eq:no-move-10-to-01} and \eqref{eq:no-move-11-to-02}.  We know that
\begin{equation}\label{eq:g-bellman-form}
    g(x,y) = c(x,y) + \frac{\rho_{\alpha}}{\gamma}(\Pop\Qop f)(x,y).
\end{equation}
Let $\bar f = \Qop f.$
Then
\begin{align}
    g(1,0)-g(0,1)
    &= \frac{\rho_{\alpha}}{\gamma}\bigl((\Pop\bar f)(1,0)-(\Pop\bar f)(0,1)\bigr) \label{eq:g10-g01-P}\\
    &= \frac{\rho_{\alpha}}{\gamma}\bigl(
        \lambda(\bar f(2,0)-\bar f(1,1))
        + \mu_1(\bar f(0,0)-\bar f(0,1))  \notag\\
    &\qquad\qquad\qquad
        + \mu_2(2\bar f(1,0)-\bar f(0,0)-\bar f(0,1))
       \bigr) \label{eq:g10-g01-expanded}\\
    &= \frac{\rho_{\alpha}}{\gamma}\bigl(
        \lambda(\bar f(2,0)-\bar f(1,1))
        + (\mu_1-\mu_2)(\bar f(0,0)-\bar f(0,1)) \notag\\
    &\qquad\qquad\qquad
        + 2\mu_2(\bar f(1,0)-\bar f(0,1))
       \bigr), \label{eq:g10-g01-grouped}
\end{align}
and
\begin{align}
    g(1,1)-g(0,2)
    &= \frac{\rho_{\alpha}}{\gamma}\bigl((\Pop\bar f)(1,1)-(\Pop\bar f)(0,2)\bigr) \label{eq:g11-g02-P}\\
    &= \frac{\rho_{\alpha}}{\gamma}\bigl(
        \lambda(\bar f(2,1)-\bar f(1,2))
        + \mu_1(\bar f(0,1)-\bar f(0,2)) \notag\\
    &\qquad\qquad\qquad
        + \mu_2(\bar f(1,0)+\bar f(1,1)-2\bar f(0,1))
       \bigr) \label{eq:g11-g02-expanded}\\
    &= \frac{\rho_{\alpha}}{\gamma}\bigl(
        \lambda(\bar f(2,1)-\bar f(1,2))
        + (\mu_1-\mu_2)(\bar f(0,1)-\bar f(0,2)) \notag\\
    &\qquad\qquad\qquad
        + \mu_2(\bar f(1,0)-\bar f(0,1)+\bar f(1,1)-\bar f(0,2))
       \bigr). \label{eq:g11-g02-grouped}
\end{align}
By the definition of $\Qop$, the function $\bar f=\Qop f$ satisfies
\begin{equation}\label{eq:fbar-20-leq-11}
    \bar f(2,0)-\bar f(1,1) \le 0,
\end{equation}
\begin{equation}\label{eq:fbar-10-leq-01}
    \bar f(1,0)-\bar f(0,1) \le 0,
\end{equation}
\begin{equation}\label{eq:fbar-21-leq-12}
    \bar f(2,1)-\bar f(1,2) \le 0,
\end{equation}
\begin{equation}\label{eq:fbar-11-leq-02}
    \bar f(1,1)-\bar f(0,2) \le 0.
\end{equation}
We also know that $\bar f\in\K$ by Lemma~\ref{lem:qclosure}, so \eqref{eq:mono-y} for $\bar f$ implies
\begin{equation}\label{eq:fbar-k4-y-monotone-at-0}
    \bar f(0,0) \le \bar f(0,1) \le \bar f(0,2).
\end{equation}
Combining \eqref{eq:g10-g01-grouped} with \eqref{eq:fbar-20-leq-11}, \eqref{eq:fbar-10-leq-01}, and \eqref{eq:fbar-k4-y-monotone-at-0} gives
\begin{equation}\label{eq:g10-leq-g01-conclusion}
    g(1,0) \le g(0,1).
\end{equation}
Combining \eqref{eq:g11-g02-grouped} with \eqref{eq:fbar-10-leq-01}, \eqref{eq:fbar-21-leq-12}, \eqref{eq:fbar-11-leq-02}, and \eqref{eq:fbar-k4-y-monotone-at-0} gives
\begin{equation}\label{eq:g11-leq-g02-conclusion}
    g(1,1) \le g(0,2).
\end{equation}
This completes the proof of Lemma~\ref{lem:vi-consistency}.
\ifarxiv

\else
\Halmos
\fi
\end{myproof}

\section{Proof for the Average-Cost Problem}\label{sec:average}

\newcommand{\sfull}{\bar{s}}
\newcommand{\Sfull}{\bar{S}}

In this section, we prove the result for the average-cost problem in Theorem~\ref{thm:main} by taking the average-cost problem as the limit of the discounted-cost problem when the discount rate $\alpha\to 0$.
This approach is standard and has been used in queueing-control models (see, e.g., \citet{LinKum_84,Haj_84}).
Throughout this section, we assume the stability condition $\lambda<\mu_1+2\mu_2$.

For the average-cost problem, since it is not known a priori that the simplified state representation $(x,y)$ in Section~\ref{sec:state} is sufficient, we consider the original state representation.
To distinguish from the simplified state, we write the state as $\sfull=(q,z,y)$, where recall that $q$ is the number of queued jobs, $z$ is the number of jobs at the fast server, and $y$ is the number of jobs at the slow servers.
Under a policy $\pi$, the long-run average cost given initial state $\sfull$ is defined as
\begin{equation}
    g^{\pi}(\sfull)=\operatorname*{\lim\sup}\limits_{T\to\infty}\frac{1}{T}\E{\int_0^{T}\left(Q(t)+Z(t)+Y(t)\right)dt \given \Sfull(0)=\sfull},
\end{equation}
where $\Sfull(t)=(Q(t),Z(t),Y(t))$ is the post-decision state at time $t$.

Our proof consists of two steps.
First, we consider a restricted version of the average-cost problem in which policies are required to be based on the simplified state representation $(x,y)$ in Section~\ref{sec:state}, i.e., policies that never let the fast server idle if there are queued jobs at the decision time.
We show in Lemma~\ref{lem:discounted-to-average} that for this restricted version of the average-cost problem, the optimal relative value function is a limit of the optimal value functions for the discounted-cost problem, and thus retains the threshold structure for an optimal policy.
In the second step, we show in Lemma~\ref{lem:opt-full-state} that the threshold policy obtained from the restricted problem is also optimal for the original average-cost problem.

\begin{lemma}\label{lem:discounted-to-average}
Consider the average-cost problem restricted to policies that never let the fast server idle if there are queued jobs.
Then there exists a sequence $(\alpha_k)$ with $\alpha_k\downarrow0$, with an associated sequence $\rho_{\alpha_k}\uparrow1$, a scalar $g$, and a function $h$, such that, for every state $s$ and a reference state $s_0=(0,0)$,
\begin{equation}
    (1-\rho_{\alpha_k})V_{\alpha_k}^{(\infty)}(s)\to g,
    \qquad
    V_{\alpha_k}^{(\infty)}(s)-V_{\alpha_k}^{(\infty)}(s_0)\to h(s),\qquad k\to\infty.
    \label{eq:discounted-to-average-limit}
\end{equation}
Further, the pair \((g,h)\) satisfies the average-cost optimality equation
\begin{equation}
    g+h(s)
    =
    c(s)+\frac{1}{\gamma}\bigl(\Pop\Qtrue h\bigr)(s),
    \qquad s\in \Zp\times\{0,1,2\},
    \label{eq:average-cost-optimality-equation}
\end{equation}
and any policy attaining the minimum in the operator
\(\Qtrue\) in \eqref{eq:average-cost-optimality-equation} is average-cost optimal with average cost \(g\).  
Consequently, the function $h\in\K$, and thus the corresponding optimal policy $\piavgopt$ has the threshold structure in Theorem~\ref{thm:main}.
\end{lemma}

The proof of \eqref{eq:discounted-to-average-limit} and \eqref{eq:average-cost-optimality-equation} in Lemma~\ref{lem:discounted-to-average} utilizes the general results in \citet{Sen_89}.
It is somewhat routine but we provide it in Appendix~\ref{sec:proof-lem-discounted-to-average}.
Once we have the limit result $V_{\alpha_k}^{(\infty)}(s)-V_{\alpha_k}^{(\infty)}(s_0)\to h(s)$ as $k\to \infty$ in \eqref{eq:discounted-to-average-limit}, noting that the cone $\K$ is closed under subtracting constants and pointwise limit, it follows that $h\in\K$.

\begin{lemma}\label{lem:opt-full-state}
The policy $\piavgopt$ in Lemma~\ref{lem:discounted-to-average} is optimal for the original average-cost problem.
Therefore, there exists an optimal policy for the original average-cost problem that has the threshold structure in Theorem~\ref{thm:main}.
\end{lemma}

The proof of Lemma~\ref{lem:opt-full-state} is provided in Appendix~\ref{sec:proof-lem-opt-full-state}.

\bibliographystyle{plainnat}
\bibliography{refs-weina}

\appendix
\renewcommand{\theHsection}{appendix.\Alph{section}}

\section{Proof of Lemma~\ref{lem:simplified-state}}\label{sec:proof-lem-simplified-state}

\begin{myproof}
We prove that if a policy $\pi$ idles the fast server after a decision when there are jobs in the queue at the decision time, then it can be strictly improved and hence cannot be optimal.

Consider the system under such a policy $\pi$, referred to as the $\pi$-system. We will construct another policy $\tpi$ based on policy $\pi$ and show a strict improvement. We refer to the system under $\tpi$ as the $\tpi$-system. Let $C(t)$ and $\tC(t)$ denote the total number of jobs at time $t$ in the $\pi$-system and the $\tpi$-system, respectively.

Consider the first time in the $\pi$-system such that the policy idles the fast server after a decision when there are jobs in the queue at the decision time.
Without loss of generality, we shift time so that this time instance is time $0$. 
There are two possible scenarios.
\begin{itemize}
\item \textbf{Scenario 1:} After the decision, the fast server is idle and there is at least one job in the queue, regardless of whether any jobs were assigned to the slow servers during the decision.
\item \textbf{Scenario 2:} After the decision, the fast server is idle and there is no job in the queue, which means that at least one job was assigned to the slow servers during the decision.
\end{itemize}
We discuss these two scenarios one by one.
Note that the coupling in this proof is up to possible relabeling of job identities, which does not affect the cost.

\subsection*{Scenario 1}

In the $\tpi$-system at time $0$, we first set the state to be the same as the post-decision state in the $\pi$-system, but we then assign a job in the queue, call it job~$J$, to the fast server. Denote the service time of job $J$ in the $\tpi$-system by $\tau_{\textup{F}}$. Then $\tau_{\textup{F}} \sim \operatorname{Exp}(\mu_1)$.
We let the two systems have the same job arrivals.

We consider the following two stopping times for the $\pi$-system.
Let $\tau_0$ be the first time the $\pi$-system assigns a job, call it job $J_0$, to the fast server.
If this never happens, let $\tau_0=\infty$.
Let $\tau_1$ be the first time the $\pi$-system assigns the last job in the queue, call it job $J_1$, to a slow server.
If this never happens, let $\tau_1=\infty$.
Then we let $\tau=\min\{\tau_0,\tau_1\}$, which is still a stopping time.
We consider the two possible cases for $\tau$, which are illustrated in Figure~\ref{fig:two-case-timeline}.

\begin{figure}[h]
\centering
\begin{tikzpicture}[x=1.05cm,y=1cm,>=Latex,font=\small]
  \coordinate (zeroTop) at (0,1.15);
  \coordinate (zeroBot) at (0,0);
  \coordinate (tauFTop) at (3,1.15);
  \coordinate (tauFBot) at (3,0);
  \coordinate (tauOneTop) at (6.5,1.15);
  \coordinate (tauOneBot) at (6.5,0);

  \draw[->] (0,1.15) -- (7.4,1.15) node[right] {time};
  \draw[->] (0,0) -- (7.4,0) node[right] {time};
  \node[left] at (0,1.15) {$\pi$-system};
  \node[left] at (0,0) {$\tpi$-system};

  \draw[thick] (0,0.98) -- (0,1.34);
  \draw[thick] (0,-0.18) -- (0,0.18);
  \node[below] at (0,-0.18) {$0$};
  \node[below] at (0,0.98) {$0$};

  \draw[thick,red] (3,-0.18) -- (3,0.18);
  \node[below,red] at (3,-0.18) {$\tau_{\textup{F}}$};
  \draw[red] (3,0) -- (3,1.15);

  \draw[thick,red] (6.5,-0.18) -- (6.5,0.18);
  \node[below,red] at (6.5,-0.18) {$\tau$};
  \draw[red] (6.5,0) -- (6.5,1.15);


  \draw[decorate,decoration={brace,amplitude=10pt,raise=0pt}] (0,1.55) -- (2.98,1.55)
    node[midway,above=8pt,text=blue] {Case 1: $\tau$ here};
  \draw[decorate,decoration={brace,amplitude=10pt,raise=0pt}] (3.02,1.55) -- (6.8,1.55)
    node[midway,above=8pt,text=blue] {Case 2: $\tau$ here};

  \draw[<->,red] (3,0.57) -- (6.5,0.57);
  \node[
    red,
    fill=white,
    inner sep=1.5pt,
    align=center
    ] at (4.75,0.57) {One fewer job\\[-0.3em] in the $\widetilde{\pi}$-system};
\end{tikzpicture}
\caption{The relative positions of $\tau$ and $\tau_{\textup{F}}$.}
\label{fig:two-case-timeline}
\end{figure}
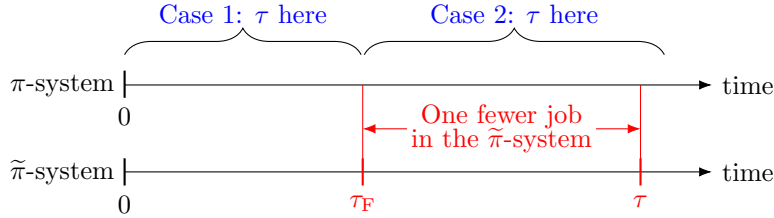

\textbf{Case 1:} $\tau \le \tau_{\textup{F}}$.\phantom{a}
In this case, first, we let the $\tpi$-system copy all the job assignments and service times from the $\pi$-system except for job $J$ until $\tau$, which can be verified to be feasible.
At time $\tau$, if $\tau=\tau_0$, we set the service time of job $J_0$ in the $\pi$-system equal to the remaining service time of job $J$ in the $\tpi$-system. 
Due to the stopping time properties, the remaining service time of job $J$ in the $\tpi$-system is still $\operatorname{Exp}(\mu_1)$. Thus the coupling is valid. 
From time $\tau$ onward, the states of the two systems are coupled to be identical.
Therefore,
\begin{equation}
\tC(t)=C(t), \qquad \forall t \ge 0.
\label{eq:case1-same-count}
\end{equation}

Now consider the other case where $\tau=\tau_1$.
At time $\tau$, the $\pi$-system assigns a queued job to a slow server but the $\tpi$-system does not have any jobs in the queue.
Again, due to the stopping time properties, the remaining service time of job $J$ at $\tau$ is still $\operatorname{Exp}(\mu_1)$.
Then in this case, we switch to the coupling for Scenario 2 at time $\tau$ and resample all the job service times accordingly.
The coupling for Scenario 2 guarantees that
\begin{equation}
\tC(t)\le C(t), \qquad \forall t \ge 0.
\label{eq:case1-switch}
\end{equation}

\textbf{Case 2:} $\tau > \tau_{\textup{F}}$.\phantom{a}
In this case, let the $\tpi$-system copy all the job assignments and service times from the $\pi$ system except for job $J$ and job $J_0$ or $J_1$.
If $\tau=\tau_0$, the $\tpi$ system simply idles the busy server during the service of job $J_0$.
If $\tau=\tau_1$, the $\tpi$ system simply idles the slow server that serves job $J_1$ in the $\pi$-system during the service of job $J_1$.
In either case,
\begin{equation}
\begin{cases}
\tC(t)=C(t)-1, & \tau_{\textup{F}} \le t < \tau,\\
\tC(t)\le C(t), & \text{otherwise}.
\end{cases}
\label{eq:case2-one-fewer-job}
\end{equation}

Combining the two cases, we write the cost difference between the two systems as follows:
\begin{align}
\E{\int_0^\infty e^{-\alpha t}\bigl(\tC(t)-C(t)\bigr)\,dt}
&\le -\E{\int_{\tau_{\textup{F}}}^{\tau} e^{-\alpha t}\,dt \given \tau>\tau_{\textup{F}}}\,\P{\tau>\tau_{\textup{F}}}.
\label{eq:discounted-cost-difference}
\end{align}
It is clear that
\begin{equation}
\E{\int_{\tau_{\textup{F}}}^{\tau} e^{-\alpha t}\,dt \given \tau>\tau_{\textup{F}}} > 0.
\end{equation}
Now since $\tau$ is a stopping time for the $\pi$-system with $\tau > 0$ with probability $1$ and $\tau_{\textup{F}}$ is sampled at time $0$ independently, we can see that
\begin{equation}
\P{\tau>\tau_{\textup{F}}}>0.
\label{eq:positive-probability-strict-improvement}
\end{equation}
Hence,
\begin{equation}
\E{\int_0^\infty e^{-\alpha t}\bigl(\tC(t)-C(t)\bigr)\,dt}<0.
\label{eq:strict-cost-improvement}
\end{equation}
Therefore, the policy $\tpi$ strictly improves over the policy $\pi$.

\subsection*{Scenario 2}
In this scenario, at time $0$, the policy $\pi$ must have assigned at least one job, call it $J$, to a slow server. 
Without loss of generality, suppose this slow server is server $2$.
We let policy $\tpi$ send job $J$ to server $1$ instead in the $\tpi$-system.

We couple the $\tpi$-system and the $\pi$-system in the following way.  We first couple the service times of job~$J$, denoted as $\tau_{\textup{F}}$ in the $\tpi$-system and $\tau_{\textup{S}}$ in the $\pi$-system. 
Let $\tau_A$ and $\tau_B$ be two independent random variables with
\begin{equation}
    \tau_A\sim \operatorname{Exp}(\mu_2),\qquad \tau_B\sim \operatorname{Exp}(\mu_1-\mu_2).
\end{equation}
Then let
\begin{equation}
    \tau_{\textup{F}} = \min\{\tau_A,\tau_B\},\qquad \tau_{\textup{S}} = \tau_A.
\end{equation}
It is easy to verify that $\tau_{\textup{F}}\sim\operatorname{Exp}(\mu_1)$ and $\tau_{\textup{S}}\sim\operatorname{Exp}(\mu_2)$.
We then couple all future job arrival times.  We also let $\tpi$ follow the same job assignment decisions as $\pi$, except for job $J$ and possibly another job discussed below.  When $\tpi$ and $\pi$ make the same job assignment, we let the service times of the job be identical in the two systems.

Let $\tau_0$ be the first time at which the $\pi$-system assigns a job, call it $J_0$, to the fast server.
If this never happens, let $\tau_0=\infty$.
We discuss three possibilities for the ordering of $\tau_0,\tau_A,\tau_B$.

\textbf{Case 1:} $\tau_0 \le \min\{\tau_A,\tau_B\}$.\phantom{a}
In this case, let $\tpi$ assign job $J_0$ to the slow server, server $2$.  Before time $\tau_0$, the two systems have the same number of jobs, except that job $J$ is at server $1$ in the $\tpi$-system and at server $2$ in the $\pi$-system.  
After the decision at time $\tau_0$, in the $\tpi$-system, job $J$ is at server $1$ and job $J_0$ is at server $2$; in the $\pi$-system, job $J_0$ is at server $1$ and job $J$ is at server $2$.
At time $\tau_0$, one can verify that conditioning on $\tau_0 \le \min\{\tau_A,\tau_B\}$, the residual service times for job $J$ in the $\tpi$-system and the $\pi$-system are still $\operatorname{Exp}(\mu_1)$ and $\operatorname{Exp}(\mu_2)$, respectively.
Therefore, we can abandon $\tau_A$ and $\tau_B$ and redraw the service times for jobs $J$ and $J_0$.
In particular, let job $J$ in the $\tpi$-system and job $J_0$ in the $\pi$-system have the same service time following $\operatorname{Exp}(\mu_1)$, and job $J_0$ in the $\tpi$-system and job $J$ in the $\pi$-system have the same service time following $\operatorname{Exp}(\mu_2)$, all independent of everything else.  Then the two systems become identical.  Therefore, in this case,
\begin{equation}
\tC(t) = C(t), \qquad \forall t \ge 0.
\label{eq:case1-count-comparison}
\end{equation}

\textbf{Case 2:} 
$\tau_A \le \tau_B$ and $\tau_0 > \tau_A$.\phantom{a}
When $\tau_A \le \tau_B$, we can see that $\tau_{\textup{F}}=\tau_{\textup{S}}$.
Therefore, the two systems are identical after $\tau_A$, and thus
\begin{equation}
\tC(t) = C(t), \qquad \forall t \ge 0.
\label{eq:case2-count-comparison}
\end{equation}

\textbf{Case 3:} 
$\tau_A > \tau_B$ and $\tau_0 > \tau_B$.\phantom{a}
In this case, the $\tpi$-system is able to copy all the job assignments except for job $J$, so the two systems have identical service times for all jobs except for job $J$.  Hence
\begin{equation}
\tC(t)
=
\begin{cases}
C(t)-1, & \tau_B \le t < \tau_A,\\
C(t), & \text{otherwise}.
\end{cases}
\label{eq:case3-count-comparison}
\end{equation}

Combining these two cases, we write the cost difference between the two systems as follows:
\begin{align}
\E{\int_0^\infty e^{-\alpha t}\bigl(\tC(t)-C(t)\bigr)\,dt}
& = -\E{\int_{\tau_B}^{\tau_A} e^{-\alpha t}\,dt \given \tau_A>\tau_B,\tau_0>\tau_B}\,\P{\tau_A>\tau_B,\tau_0>\tau_B}.
\label{eq:discounted-cost-difference-scenario2}
\end{align}
It is clear that
\begin{equation}
    \E{\int_{\tau_B}^{\tau_A} e^{-\alpha t}\,dt \given \tau_A>\tau_B,\tau_0>\tau_B} > 0.
\end{equation}
It remains to show that $\P{\tau_A>\tau_B,\tau_0>\tau_B}>0$.
Note that if there are no job arrivals before $\tau_B$, then the $\pi$ system cannot assign a job to the fast server before $\tau_B$.
Let $\tau_C\sim \operatorname{Exp}(\lambda)$ be independent of everything else.
Then
\begin{align}
\P{\tau_A>\tau_B,\tau_0>\tau_B} \ge \P{\tau_A>\tau_B,\tau_C>\tau_B} >0.
\end{align}
Hence,
\begin{equation}
\E{\int_0^\infty e^{-\alpha t}\bigl(\tC(t)-C(t)\bigr)\,dt}<0.
\label{eq:strict-cost-improvement}
\end{equation}
Therefore, the policy $\tpi$ strictly improves over the policy $\pi$, which completes the proof.

Combining these two scenarios completes the proof of the lemma.
\ifarxiv

\else
\Halmos
\fi
\end{myproof}

\section{Proof of Lemma~\ref{lem:discounted-to-average}}\label{sec:proof-lem-discounted-to-average}
\newcommand{\Tstd}{\mathcal{T}^{\mathrm{standard}}}
\newcommand{\A}{\mathcal{A}}
\newcommand{\Ecal}{\mathcal{E}}
\renewcommand{\tpi}{\widetilde{\pi}}
\newcommand{\tS}{\widetilde{S}}
\newcommand{\tX}{\widetilde{X}}
\newcommand{\tY}{\widetilde{Y}}

\begin{myproof}
We use standard results for Markov decision processes (MDPs) with a countable state space in \citet{Sen_89}.

We first explain how our Bellman operator $\Bellman$ defined in \eqref{eq:Bellman} for the post-decision-state value function can be written as the Bellman operator for a standard-form MDP, as in \citet{Put_05} and as used in \citet{Sen_89}.  The Bellman operator for a standard-form MDP with discount factor $\rho_\alpha$ is typically written as
\begin{equation}
\label{eq:standard-bellman}
(\Tstd f)(s)
= \min_{a\in \A(s)} \left\{ c(s,a) + \rho_\alpha \sum_{s'} P(s'\mid s,a) f(s') \right\},
\end{equation}
where $s$ is the current state, $a$ is an action, $\A(s)$ is the set of feasible actions at state $s$, $c(s,a)$ is the one-step cost, and $P(s'\mid s,a)$ is the transition probability from state $s$ to state $s'$ under action $a$.

For our problem, we define a standard-form MDP as follows.  
The state is still $s=(x,y)$, and the cost function is still $c(s)=x+y$.
Define an action $a$ as a vector $a=(a_e)_{e\in\Ecal}$,
where the index $e$ corresponds to a type of event.  Specifically, $e$ is from the following set of events
\begin{equation}
\label{eq:event-set}
\begin{aligned}
\Ecal = \{&\text{arrival},\ \text{potential departure from server }1,\\
&\text{potential departure from server }2,\
\text{potential departure from server }3\}.
\end{aligned}
\end{equation}
Each component $a_e$ can be thought of as what to do if event $e$ happens.  For each state $s$, let $\sigma(s,e)$ be the state immediately after event $e$.  Then define the set of feasible actions for state $s$ by
\begin{equation}
\label{eq:standard-action-set}
\A(s)
= \left\{ a=(a_e)_{e\in\Ecal}: \text{for all } e\in\Ecal,\ a_e \text{ is a feasible job assignment for } \sigma(s,e) \right\}.
\end{equation}
Let $\psi(\sigma(s,e),a_e)$ be the post-decision state immediately after applying job assignment $a_e$ to state $\sigma(s,e)$.  Now define
\begin{equation}
\label{eq:standard-transition-probability}
P(s'\mid s,a)
= \sum_{e\in\Ecal} p_e \1\left\{\psi(\sigma(s,e),a_e)=s'\right\},
\end{equation}
where $p_e$ is the probability that the next event is $e$.  It is not hard to see that this is a valid transition probability matrix.  
Also note that our cost function $c$ does not depend on the action.  Therefore, the standard Bellman operator can be written as
\begin{align}
\label{eq:standard-to-post-decision-operator}
(\Tstd f)(s)
&= \min_{a=(a_e)\in\A(s)} \left\{ c(s) + \rho_\alpha \sum_{s'} P(s'\mid s,a) f(s') \right\} \\
&= c(s) + \rho_\alpha \min_{a=(a_e)\in\A(s)}
\sum_{s'} \sum_{e\in\Ecal} p_e
\1\left\{\psi(\sigma(s,e),a_e)=s'\right\} f(s') \\
&= c(s) + \rho_\alpha \sum_{e\in\Ecal} p_e
\min_{a_e:\ a_e \text{ feasible for } \sigma(s,e)}
\sum_{s'} \1\left\{\psi(\sigma(s,e),a_e)=s'\right\} f(s') \\
&= c(s) + \rho_\alpha \sum_{e\in\Ecal} p_e
\min_{a_e:\ a_e \text{ feasible for } \sigma(s,e)}
f\bigl(\psi(\sigma(s,e),a_e)\bigr) \\
&= (\Bellman f)(s).
\end{align}
Thus our Bellman operator $\Bellman$ is the same as $\Tstd$ for the standard-form MDP defined above, and our discounted optimal value function $V_\alpha^{(\infty)}$ is also the value function for the associated standard-form MDP. 
This assures that although we used the post-decision state and the corresponding Bellman operator, it still fits into standard-form MDPs, and thus we can use tools developed for standard-form MDPs.

We are now ready to verify the assumptions and apply the theorem in \citet{Sen_89}, which directly implies Lemma~\ref{lem:discounted-to-average}.  The assumptions we verify for our problem are the following.

\paragraph{Assumption 1 in \citet{Sen_89}.}
\emph{
For every state $s$ and every $\alpha>0$,
\begin{equation}
\label{eq:sennott-assumption-finite-value}
V_\alpha^{(\infty)}(s)<\infty.
\end{equation}
}

This follows from the fact that jobs arrive according to a Poisson process.  We have used this fact when defining $V_\alpha^{(\infty)}$ in Section~\ref{sec:value-iteration}.

\paragraph{Assumption 2 in \citet{Sen_89}.}
\emph{
There exist a reference state $s_0$ and a nonnegative constant $N$ such that
\begin{equation}
\label{eq:sennott-assumption-lower-bound}
V_\alpha^{(\infty)}(s)-V_\alpha^{(\infty)}(s_0) \ge -N,
\qquad \forall s,\ \forall \alpha>0.
\end{equation}
}

To see why this is true, choose $s_0$ to be the state $(0,0)$, that is, the state where the system has no jobs.  Then, due to the monotonicity properties \eqref{eq:mono-x} and \eqref{eq:mono-y} established for $V_\alpha^{(\infty)}$, we know
\begin{equation}
\label{eq:monotonicity-gives-lower-bound}
V_\alpha^{(\infty)}(s)-V_\alpha^{(\infty)}(s_0)
\ge 0,
\qquad \forall s,\ \forall \alpha>0.
\end{equation}
Thus Assumption 2 holds with $s_0=(0,0)$ and $N=0$.

\paragraph{Assumption 3 in \citet{Sen_89}.}
\emph{There exist a reference state $s_0$ and nonnegative constants $M_s$ such that
\begin{equation}
\label{eq:sennott-assumption-upper-bound}
V_\alpha^{(\infty)}(s)-V_\alpha^{(\infty)}(s_0) \le M_s,
\qquad \forall s,\ \forall \alpha>0.
\end{equation}
In addition,
\begin{equation}
\label{eq:sennott-assumption-integrability}
\sum_{s'} P(s'\mid s,a) M_{s'} < \infty,
\qquad \forall s,\ \forall a\in \A(s).
\end{equation}
}

To verify this assumption for our problem, we consider the work-conserving policy $\tpi$, that is, the policy that first uses the fast server whenever possible and also uses the slow servers whenever possible.
Below we first establish a few basic properties for the system under $\tpi$, and then we define nonnegative constants $M_s$ using quantities based on this system.
To simplify the argument, consider the uniformization of the continuous-time Markov chain under $\tpi$ with uniformization rate
\begin{equation}
\label{eq:uniformization-rate}
\gamma = \lambda + \mu_1 + 2\mu_2.
\end{equation}
That is, consider the discrete-time Markov chain whose state at time step $n$ is the post-decision state after the $n$-th event, denoted as $\tS(n) = \bigl(\tX(n),\tY(n)\bigr).$
Consider the following set of states, which consists of all states that communicate with the state $s_0=(0,0)$:
\begin{equation}
\label{eq:communicating-class-R}
\mathcal R
= \{(0,0),(1,0),(0,1),(1,1),(0,2)\}
  \cup \{(x,2): x\ge 1\}
\triangleq \mathcal B \cup \{(x,2): x\ge 1\}.
\end{equation}
States not in $\mathcal R$ are in the set
\begin{equation}
\label{eq:transient-set-Rc}
\mathcal R^c
= \{(x,y)\in \mathbb Z_+\times\{0,1,2\}: x\ge 2,\ y\in\{0,1\}\}.
\end{equation}
States in $\mathcal R^c$ are transient since, starting from any $s\in\mathcal R^c$, the state transitions to a state in $\mathcal R$ after one step due to the work-conserving decision. Note that the Markov chain $(\tS(n):n\in\mathbb Z_+)$ is aperiodic, since the state $(0,0)$ has a self-transition under uniformization and all states in $\mathcal{R}$ communicate.

We now show that $(\tS(n))$ is positive recurrent in $\mathcal R$. Consider the cost function $c(x,y)=x+y$ as the Lyapunov function. Then for any $(x,y)\in\{(x,2):x\ge 1\}$, the drift of $c$ is
\begin{align}
\Delta c(x,y)
&\triangleq
\E{\tX(n+1)+\tY(n+1)-\tX(n)-\tY(n)
\given (\tX(n),\tY(n))=(x,y)} \notag\\
&= \frac{\lambda}{\gamma}(x+y+1)
  + \frac{\mu_1+2\mu_2}{\gamma}(x+y-1)
  - (x+y) \notag\\
&= \frac{1}{\gamma}(\lambda-\mu_1-2\mu_2) <0.
\label{eq:negative-drift-tail}
\end{align}
We also know that for any $(x,y)\in\mathcal B$,
\begin{equation}
\label{eq:bounded-drift-finite-set}
\Delta c(x,y)\le 1.
\end{equation}
Since $\mathcal B$ is a finite set, by the Foster-Lyapunov theorem, the Markov chain $(\tS(n))$ is positive recurrent in $\mathcal R$. As a result, $(\tS(n))$ has a unique stationary distribution with a positive probability on each state in $\mathcal R$.

We can also show that
\begin{equation}
\label{eq:stationary-finite-first-moment}
\E{c\bigl(\tX(\infty),\tY(\infty)\bigr)}<\infty,
\end{equation}
where $(\tX(\infty),\tY(\infty))$ follows the stationary distribution. This follows from the drift bounds above and the fact that
\begin{equation}
\label{eq:bounded-one-step-increments}
c\bigl(\tX(n+1),\tY(n+1)\bigr)-c\bigl(\tX(n),\tY(n)\bigr)\le 1 \qquad \text{a.s.},
\end{equation}
using the results in \citet{Haj_82}.

We are now ready to define the needed constants $M_s$ for each state $s$. Starting the chain $(\tS(n))$ from $s$, let $n_s$ be the first time the state hits $s_0=(0,0)$. That is,
\begin{equation}
\label{eq:hitting-time-ns}
n_s = \inf\{n\in\mathbb Z_+: \tS(n)=(0,0)\}.
\end{equation}
Let
\begin{equation}
\label{eq:Ms-definition}
M_s = \E{\sum_{n=0}^{n_s-1} c\bigl(\tX(n),\tY(n)\bigr)
\given (\tX(0),\tY(0))=s}.
\end{equation}
It is easy to see that $M_s$ does not depend on the discount rate $\alpha$. We now show that
\begin{equation}
\label{eq:Ms-finite-goal}
M_s<\infty, \qquad \forall s.
\end{equation}
It suffices to focus on $s\in\mathcal R$, since we know that even if $s\notin\mathcal R$, the state at time step $1$ will be in $\mathcal R$. When $s\in\mathcal R$, we know that $(\tS(n))$ on $\mathcal R$ is aperiodic, irreducible, and positive recurrent. We also know that $\E{c\bigl(\tX(\infty),\tY(\infty)\bigr)}<\infty$.
By Proposition 4 in \citet{Sen_89}, we know that $M_s<\infty$.

We now verify that
\begin{equation}
\label{eq:upper-bound-by-Ms-goal}
V_\alpha^{(\infty)}(s)-V_\alpha^{(\infty)}(s_0)
\le M_s, \qquad \forall s,\ \forall \alpha>0.
\end{equation}
To see this, consider the policy that runs $\tpi$ until the first time the state hits $(0,0)$ and then switches to the optimal discounted policy.
Then for any state $s$ and any $\alpha>0$
\begin{align}
V_\alpha^{(\infty)}(s)
&\le \E{\sum_{n=0}^{n_s-1} \rho_\alpha^n c\bigl(\tX(n),\tY(n)\bigr)+ \rho_\alpha^{n_s} V_\alpha^{(\infty)}(s_0)
\given (\tX(0),\tY(0))=s}\nonumber\\
&\le M_s+V_\alpha^{(\infty)}(s_0).
\label{eq:discounted-value-upper-bound-by-Ms}
\end{align}

In addition, for each state $s$, there are only finitely many states $s'$ such that $P(s'\mid s,a)>0$ for some action $a$.  Therefore,
\begin{equation}
\label{eq:finite-transition-weighted-Ms}
\sum_{s'} P(s'\mid s,a)M_{s'}<\infty,
\qquad \forall s,\ \forall a\in \A(s).
\end{equation}
This completes the verification of Assumption 3 in \citet{Sen_89}.

Now we have verified all the assumptions needed.  Applying the theorem in \citet{Sen_89} gives the desired result.
The theorem is proved in the appendix of \citet{Sen_89}, which has explicit limit relationships that give \eqref{eq:discounted-to-average-limit} in Lemma~\ref{lem:discounted-to-average}.
\ifarxiv

\else
\Halmos
\fi
\end{myproof}

\section{Proof of Lemma~\ref{lem:opt-full-state}}\label{sec:proof-lem-opt-full-state}

\begin{myproof}
Let $\piavgopt$ be a stationary policy attaining the minimum in the average-cost optimality equation in Lemma~\ref{lem:discounted-to-average}. 
By Lemma~\ref{lem:discounted-to-average}, $\piavgopt$ is average-cost optimal for the restricted problem on the reduced state space and has long-run average cost $g$.
We define the corresponding policy in the original state space as follows. 
With a slight abuse of notation, we also call this policy $\piavgopt$. 
At each decision time, let the state be $\sfull=(q,z,y)\in \mathbb{Z}_{+}\times\{0,1\}\times\{0,1,2\}$ in the original state representation.
If the fast server is idle and the queue is nonempty, the policy first assigns one queued job to the fast server. It then applies the slow-server decision prescribed by the reduced-state policy $\piavgopt$ at the simplified state $(q+z,y)$.
Then under policy $\piavgopt$, the long-run average cost starting from any state $\sfull$ is $g^{\piavgopt}(\sfull)=g$.
It remains to show that no admissible policy for the original problem can have
average cost smaller than $g$.

Consider any admissible policy $\pi$ for the original problem and an arbitrary initial state $\sfull=(q, z, y)$. 
Note that here the initial state $\sfull$ is a post-decision state, and the first decision will be made after the first event.
Let $s$ be the corresponding simplified state $s=(q+z, y)$.
Let $g^{\pi}(\sfull)$ be the long-run average cost of $\pi$ starting from the state $\sfull$.
We also consider the total discounted cost of $\pi$ starting from $\sfull$ under a discount rate $\alpha>0$, denoted as $\bar{V}^{\pi}_{\alpha}(\sfull)$.
Our proof has two steps.
In the first step, we prove that
\begin{equation}
    g
    \le
    \limsup_{\alpha\downarrow0}
    \alpha \bar{V}^{\pi}_{\alpha}(\sfull).
    \label{eq:g-bounded-by-policy-discounted-limsup-app-c-pre}
\end{equation}
In the second step, we prove that
\begin{equation}
    \limsup_{\alpha\downarrow0}
    \alpha \bar{V}^{\pi}_{\alpha}(\sfull)
    \le
    g^{\pi}(\sfull).
    \label{eq:abelian-inequality-app-c}
\end{equation}
Then combining them gives the desired result.

We start with the first step.
Let $\bar{V}^{*}_{\alpha}(\sfull)$ be the optimal value function for the $\alpha$-discounted-cost problem in the original state space.
Then $\bar{V}^{*}_{\alpha}(\sfull)\le \bar{V}^{\pi}_{\alpha}(\sfull)$.
We now relate $\bar{V}^{*}_{\alpha}(\sfull)$ to $\Val(s)$.
For the full post-decision state
\(\sfull=(q,z,y)\), define its fast-filled representation
\[
\sfull^F =
\begin{cases}
(q,z,y), & z=1\text{ or }q=0,\\
(q-1,1,y), & z=0,\ q\ge1.
\end{cases}
\]
Recall that the corresponding simplified state is $s=(q+z,y)$.
It is not hard to see that $\bar{V}^{*}_{\alpha}(\sfull^F)=V_\alpha^*(s)$.
Moreover,
\[
\bar V_\alpha^*(\bar s^F)\le \bar V_\alpha^*(\bar s),
\]
which can be proven using the same coupling arguments in the proof of Lemma~\ref{lem:simplified-state}. Therefore, $V_\alpha^*(s)\le \bar{V}^{*}_{\alpha}(\sfull)\le \bar{V}^{\pi}_{\alpha}(\sfull)$.
Recall that $\Val(s)=(\alpha+\gamma)V_\alpha^*(s)$, and $\rho_\alpha = \frac{\gamma}{\alpha+\gamma}$.
Thus,
\begin{equation}
    (1-\rho_{\alpha})V^{(\infty)}_{\alpha}(s)
    \le
    \alpha \bar{V}^{\pi}_{\alpha}(\sfull).
    \label{eq:discounted-lower-bound-scaled-app-c}
\end{equation}
Now take the sequence $\alpha_k\downarrow0$ from Lemma~\ref{lem:discounted-to-average}. Since
\begin{equation}
    (1-\rho_{\alpha_k})V^{(\infty)}_{\alpha_k}(s)\to g,
    \label{eq:reduced-vanishing-discount-limit-app-c}
\end{equation}
we obtain
\begin{equation}
    g
    \le
    \limsup_{k\to\infty}
    \alpha_k \bar{V}^{\pi}_{\alpha_k}(\sfull)\le
    \limsup_{\alpha\downarrow 0}
    \alpha \bar{V}^{\pi}_{\alpha}(\sfull).
    \label{eq:g-bounded-by-policy-discounted-limsup-app-c}
\end{equation}

We now carry out the second step, where we prove \eqref{eq:abelian-inequality-app-c}.
The arguments are standard.
Let $(Q(t),Z(t),Y(t))$ be the system state at time $t$ under policy $\pi$, and let
\begin{equation}
    \bar{c}^{\pi}(t)
    =
    \E{Q(t)+Z(t)+Y(t)\given (Q(0),Z(0),Y(0))=\sfull}.
    \label{eq:expected-population-process-app-c}
\end{equation}
Then
\begin{equation}
    \bar{V}^{\pi}_{\alpha}(\sfull)
    =
    \int_{0}^{\infty} e^{-\alpha t}\bar{c}^{\pi}(t)\,dt.
    \label{eq:discounted-cost-as-laplace-transform-app-c}
\end{equation}
Let 
\begin{equation}
    \Gamma^{\pi}(t)
    =
    \int_{0}^{t}\bar{c}^{\pi}(u)\,du.
    \label{eq:cumulative-expected-cost-app-c}
\end{equation}
We know that $\bar{c}^{\pi}(t)$ is nonnegative and locally integrable since $\bar{c}^{\pi}(t)\le q + z + y + \lambda t$ due to the Poisson arrivals.
Then $\Gamma^{\pi}(t)$ is absolutely continuous and nondecreasing.

We know that
\begin{equation}
    g^{\pi}(\sfull)
    =
    \limsup_{T\to\infty}
    \frac{1}{T}\Gamma^{\pi}(T).
    \label{eq:average-cost-limsup-app-c}
\end{equation}
If $g^{\pi}(\sfull)=\infty$, then there is nothing to prove. Otherwise, fix any
$\varepsilon>0$. By \eqref{eq:average-cost-limsup-app-c}, there exists
$T_0<\infty$ such that
\begin{equation}
    \Gamma^{\pi}(t)
    \le
    \bigl(g^{\pi}(\sfull)+\varepsilon\bigr)t,
    \qquad t\ge T_0.
    \label{eq:limsup-eventual-average-bound-app-c}
\end{equation}
The properties of $\Gamma^{\pi}(t)$ allow the following integration by parts
\begin{equation}
\alpha \bar{V}^{\pi}_{\alpha}(\sfull)
=
\alpha\int_{0}^{\infty}e^{-\alpha t}\,d\Gamma^{\pi}(t)
=
\alpha^2\int_{0}^{\infty}e^{-\alpha t}\Gamma^{\pi}(t)\,dt,
\label{eq:abelian-integration-by-parts-app-c}
\end{equation}
where the boundary term vanishes because $\Gamma^{\pi}(t)=O(t^2)$.
Using
\eqref{eq:limsup-eventual-average-bound-app-c}, we obtain
\begin{equation}
\begin{aligned}
    \alpha \bar{V}^{\pi}_{\alpha}(\sfull)
    &=
    \alpha^2\int_{0}^{T_0}e^{-\alpha t}\Gamma^{\pi}(t)\,dt
    +
    \alpha^2\int_{T_0}^{\infty}e^{-\alpha t}\Gamma^{\pi}(t)\,dt  \\
    &\le
    \alpha^2\int_{0}^{T_0}e^{-\alpha t}\Gamma^{\pi}(t)\,dt
    +
    \bigl(g^{\pi}(\sfull)+\varepsilon\bigr)
    \alpha^2\int_{T_0}^{\infty}e^{-\alpha t}t\,dt .
\end{aligned}
    \label{eq:abelian-upper-bound-decomposition-app-c}
\end{equation}
Taking the $\limsup_{\alpha\downarrow 0}$ gives
\begin{equation}
    \limsup_{\alpha\downarrow0}
    \alpha \bar{V}^{\pi}_{\alpha}(\sfull)
    \le
    g^{\pi}(\sfull)+\varepsilon.
    \label{eq:abelian-bound-with-epsilon-app-c}
\end{equation}
Letting $\varepsilon\downarrow0$ gives
\begin{equation}
    \limsup_{\alpha\downarrow0}
    \alpha \bar{V}^{\pi}_{\alpha}(\sfull)
    \le
    g^{\pi}(\sfull).
    \label{eq:abelian-inequality-app-c-2}
\end{equation}
This completes the proof.
\ifarxiv

\else
\Halmos
\fi
\end{myproof}

\section{Brief Report on Interactions with GPT-5.5 Pro}
\label{sec:report-gpt}

In this section, we provide a brief report on some of our interactions with GPT-5.5 Pro, also referred to as the \emph{model}.
The purpose of this section is to give context for how the results in this paper were obtained and to document the role played by the model in the discovery and writing process.
The prompting was carried out by one of the authors, Weina Wang.
All prompts reported here were quoted verbatim.

We begin with the initial prompt and the corresponding initial response from GPT-5.5 Pro, where the model already found the key proof ideas.
We next give some examples of the errors and inaccuracies in the model's generated content.
We then present two additional prompt examples, together with summaries of the model's responses. 
The goal of this section is not to give a complete transcript of our use of AI, but rather to demonstrate some author-AI interactions and the kinds of reasoning produced by the model.

It is important to note that, after many rounds of interaction, the writeup produced by GPT-5.5 Pro was mostly correct at the level of the main mathematical ideas.
However, its overall quality was still far from the standard expected of a theoretical paper in computer science or operations research. 
For example, it introduced many redundant notations, the logical flow was often indirect, and some arguments were still too brief to be verifiable. 
The present paper is therefore a complete human rewriting of the writeup generated by the model.

\subsection{First Prompt and First Response}

We started by asking GPT-5.5 Pro to solve the problem and write a paper about it, giving it some background materials and references. The first prompt and response are given below.

\begin{promptbox}[title={First prompt}]
In this project, I'd like to work on an old open problem in queueing. This is a well-known problem that has been studied by many people, but it's still unsolved. The position paper I uploaded to the sources named ``the slower-server problem with multiple slower servers-100-views-on-queues-Ger Koole.pdf" gives a good articulation of the problem. We can focus on the case where there are one fast servers and two slower servers with the same rate. I'm also uploading some classical papers on this topic for the setting where there is one fast and one slow server: ``A simple proof of the optimality of a threshold policy in a two-server queueing systems.pdf", ``A note on `Optimal control of a queuing system with two heterogeneou servers' ". Those are just examples of papers in this field. Feel free to read other papers. The position paper also has good pointers in it. Can you solve this problem?
\end{promptbox}

\begin{chatgptbox}[title={First GPT response (scribed verbatim into LaTeX, citations omitted)}]
Thought for 20m 39s
\medskip

I think the one-fast/two-identical-slow \textbf{Poisson-arrival} case can be solved by strengthening Koole’s value-iteration cone. The key is to add one inequality that is not present in the two-server proof: convexity in the number of busy slow servers \textbf{along fixed total population}. That inequality makes the two-slow-server decision operator closed.

Koole's position note states the conjecture exactly in the form ``if sending a customer to a slow server is optimal in \((x,y)\), then it is also optimal in \((x+1,y)\)'' for one fast server and two equal slow servers, and notes that the usual inequality-propagation method breaks for multiple slow servers. [source cited] The two-server proof propagates multimodularity plus monotonicity. [source cited] The incomplete-proof problem flagged by de Véricourt–Zhou is that earlier multi-server arguments missed operator cases/boundaries, so the proof below is designed around an explicit closed cone rather than a coupling shortcut. [source cited]

\section*{Proposed theorem}

For an \(M/M/1+2\) system with one fast server of rate \(\mu_1\), two identical slow servers of rate \(\mu_2<\mu_1\), nonpreemptive assignment, linear holding cost equal to the number of customers in system, and \(\lambda<\mu_1+2\mu_2\), there is an optimal discounted and average-cost stationary policy with threshold form:

\[
\text{if a slow-server assignment is optimal at }(x,y),\text{ then it is optimal at }(x+1,y),
\]

for \(y=0,1\). More strongly, for \(y=0\) there are ordered thresholds for using zero, one, or two slow servers.

\section*{The Bellman setup}

Uniformize so that \(\lambda+\mu_1+2\mu_2=1\). Let \(x\) be the number of customers in the queue plus the customer at the fast server, and let \(y\in\{0,1,2\}\) be the number of busy slow servers. Write \(f_y(x)=f(x,y)\).

Use the post-decision value recursion

\[
V_{n+1}=c+P\,Q V_n,\qquad c(x,y)=x+y,
\]

where \(Q\) is the slow-server assignment operator and \(P\) is the event operator. For the proof, use the relaxed assignment operator

\[
(Qf)(x,y)=\min_{0\le a\le 2-y} f(x-a,y+a),
\]

with infeasible negative \(x-a\) terms treated as \(+\infty\). This relaxation allows “moving” a fast-server customer to a slow server at the boundary, but such actions are dominated by keeping that customer on the fast server. This is the same fastest-available-server principle proved by coupling in the older literature; Weber states the general result that, when an assignment is made, an optimal policy assigns to the fastest available processor. [source cited]
Walrand’s two-server note also uses this type of pathwise argument. [source cited] Thus the relaxed and original models have the same optimal value, with a feasible optimal selector.

\section*{The invariant cone}

Define a cone \(\mathcal C\) of functions \(f:\mathbb N_0\times\{0,1,2\}\to\mathbb R\) satisfying the following inequalities.

For all valid \(x\) and \(y\):

\[
\Delta_x f_y(x)=f_y(x+1)-f_y(x)\ge 0,
\]

\[
\Delta_x^2 f_y(x)=f_y(x+1)-2f_y(x)+f_y(x-1)\ge 0,
\]

\[
f_{y+1}(x)-f_y(x)\ge 0,
\]

\[
f_0(x)+f_2(x)-2f_1(x)\ge 0,
\]

\[
S_y(x):=f_y(x)+f_{y+1}(x+1)-f_y(x+1)-f_{y+1}(x)\ge 0,
\]

\[
M_y(x):=f_y(x+1)+f_{y+1}(x-1)-f_y(x)-f_{y+1}(x)\ge 0,
\]

and the new fixed-total-population convexity inequality

\[
R(x):=f_0(x)+f_2(x-2)-2f_1(x-1)\ge 0.
\]

The last inequality is the crucial addition. In total-population coordinates \(n=x+y\), it says the value is convex in the number of busy slow servers along the line \(n=\text{constant}\). This is exactly the missing curvature that makes the two-slow-server minimum operator behave.

\section*{Why the cone is closed}

The event operator is

\[
(Pf)(x,0)=\lambda f_0(x+1)+\mu_1 f_0((x-1)^+)+2\mu_2 f_0(x),
\]

\[
(Pf)(x,1)=\lambda f_1(x+1)+\mu_1 f_1((x-1)^+)+\mu_2 f_0(x)+\mu_2 f_1(x),
\]

\[
(Pf)(x,2)=\lambda f_2(x+1)+\mu_1 f_2((x-1)^+)+2\mu_2 f_1(x).
\]

The nontrivial checks reduce to nonnegative combinations of the cone inequalities. For example, away from the \(x=0\) boundary,

\[
M_0(Pf)(x)
=
\lambda M_0(f)(x+1)
+\mu_1 M_0(f)(x-1)
+\mu_2\{\Delta_x^2 f_0(x)+M_0(f)(x)\}\ge 0,
\]

\[
M_1(Pf)(x)
=
\lambda M_1(f)(x+1)
+\mu_1 M_1(f)(x-1)
+\mu_2\{\Delta_x^2 f_1(x)+M_0(f)(x)\}\ge 0,
\]

and

\[
R(Pf)(x)
=
\lambda R(f)(x+1)
+\mu_1 R(f)(x-1)
+2\mu_2 M_0(f)(x-1)\ge 0.
\]

The \(S_y\) inequalities are similarly closed:

\[
S_0(Pf)(x)=
\lambda S_0(f)(x+1)+\mu_1 S_0(f)(x-1)+\mu_2S_0(f)(x),
\]

\[
S_1(Pf)(x)=
\lambda S_1(f)(x+1)+\mu_1 S_1(f)(x-1)+\mu_2S_0(f)(x).
\]

The coordinate monotonicity, coordinate convexity, and same-\(x\) slow-server convexity inequalities are direct positive-combination checks. Boundary cases use the same inequalities with the fictitious fast-completion term \(f_y((x-1)^+)\).

For the decision operator \(Q\), transform to total-population coordinates:

\[
H(n,j)=f(n-j,j),\qquad j=0,1,2.
\]

Then \(Q\) becomes the tail-minimum operator

\[
\bar H(n,j)=\min_{k\ge j}H(n,k).
\]

Let

\[
A_n=H(n,0)-H(n,1),\qquad B_n=H(n,1)-H(n,2).
\]

The \(M_y\) inequalities imply \(A_n\) and \(B_n\) are nondecreasing in \(n\). The new \(R\)-inequality gives

\[
A_n\ge B_n.
\]

Therefore the minimizer of \(\{H(n,0),H(n,1),H(n,2)\}\) has the ordered form

\[
H(n,0)\;\to\;H(n,1)\;\to\;H(n,2)
\]

as \(n\) increases; it cannot oscillate. This ordered-tail-min property is what fails without \(R\). A direct three-case verification then shows \(Qf\in\mathcal C\). Thus

\[
f\in\mathcal C \quad\Longrightarrow\quad Qf\in\mathcal C \quad\Longrightarrow\quad P Qf\in\mathcal C.
\]

Since \(c\in\mathcal C\), all finite-horizon and discounted value iterates stay in \(\mathcal C\), and the discounted fixed point is in \(\mathcal C\). The average-cost bias follows by the usual vanishing-discount argument because the cone is closed under pointwise limits and additive normalization.

\section*{Threshold conclusion}

Let \(V\) be the discounted value function or an average-cost bias function. Define the diagonal differences

\[
D_0(x)=V_0(x)-V_1(x-1),
\]

\[
D_1(x)=V_1(x)-V_2(x-1).
\]

The \(M_y\) inequalities say precisely that \(D_0(x)\) and \(D_1(x)\) are nondecreasing in \(x\).

For \(y=1\), using the remaining idle slow server is optimal exactly when

\[
V_2(x-1)\le V_1(x)
\quad\Longleftrightarrow\quad
D_1(x)\ge 0.
\]

Since \(D_1(x)\) is nondecreasing, this is an upper set in \(x\).

For \(y=0\), using at least one slow server is optimal when

\[
\min\{V_1(x-1),V_2(x-2)\}\le V_0(x).
\]

Equivalently,

\[
D_0(x)\ge 0
\quad\text{or}\quad
D_0(x)+D_1(x-1)\ge 0.
\]

Both \(D_0(x)\) and \(D_0(x)+D_1(x-1)\) are nondecreasing in \(x\), so this condition is also an upper set. The \(R\)-inequality further gives the ordered structure for choosing one versus two slow servers.

So the conjecture follows for the one-fast/two-identical-slow Poisson model.

I would treat this as a proof draft rather than a final paper proof: the main item to write out fully is the finite three-case verification that \(Q\) preserves the cone. The argument above is the part I’d build the note around. The MAP-arrival version should be approachable by adding an arrival-phase index and requiring the same inequalities componentwise, but I have not checked that extension yet.
\end{chatgptbox}

\subsection{Errors and Inaccuracies}
The author found a minor algebraic error in early versions of the proofs generated by GPT-5.5 Pro. 
In addition, the early versions of the proof of Lemma~\ref{lem:simplified-state} were inaccurate and informal. 
Even after several iterations, it seemed difficult to make GPT-5.5 Pro produce a fully precise proof of this lemma, which involves stochastic coupling. 
However, when asked to check the correctness of this proof, the model was able to identify logical gaps, including gaps in the enumeration of possible scenarios.

\begin{promptbox}[title={Prompt that caught a minor algebraic error}]
I'm reading the proof of Lemma 3 and realized some issues. For example, when you give the expression of X_0^g(x), the second term doesn't seem to be correct when x = 0. Can you check?
\end{promptbox}

The model agreed that this was an error and revised the argument. 
This error occurred in the application of the fast-departure operator $\DFop$ in the proof of Lemma~\ref{lem:pclosure}, specifically in \eqref{eq:event-fast-L1}. The correct algebra is given in the present version of \eqref{eq:event-fast-L1}, where $L_1^g(x,y)=0$ when $x=0$. In the earlier version referred to in the prompt, the corresponding term was
$h((x-1)^+ +1,y)-h((x-1)^+,y)$, which was incorrect since it is equal to $h(1,y)-h(0,y)$ when $x=0$.

Below we give an example of an earlier version of Lemma~\ref{lem:simplified-state} and its proof.
It has vague descriptions such as ``using dummy customers only to synchronize states after the earlier of the two coupled service completions'', without specifying what dummy customers are and how the synchronization is constructed.

\newtheorem*{manuallemma}{Lemma 1 (Fast-server reduction)}
\begin{chatgptbox}[title={Earlier version of GPT generated Lemma~\ref{lem:simplified-state} and its proof}]
\begin{manuallemma}
For the original nonpreemptive system, there is an optimal policy that never routes a customer to an idle slow server while the fast server is idle and a customer is available for service.
\end{manuallemma}

\begin{proof}
This is the standard fastest-available-server interchange.  Suppose a policy assigns a customer $C$ to a slow server at an instant at which the fast server is idle.  Construct a modified policy that assigns $C$ to the fast server instead and then mimics the original policy, using dummy customers only to synchronize states after the earlier of the two coupled service completions.  Couple the first fast and slow service times by taking a single exponential random variable $E$ of rate one and setting
\begin{equation*}
        \sigma_1=E/\mu_1,\qquad \sigma_2=E/\mu_2 .
\end{equation*}
Since $\mu_1>\mu_2$, $\sigma_1\le \sigma_2$ almost surely.  The modified system therefore completes $C$ no later than the original system.  Up to the synchronization time it has no more real customers than the original system; after synchronization the two systems can be run with the same future arrivals, services, and decisions.  Dummy customers, if introduced, are not charged holding cost and are used only to preserve the same future service-clock construction.  Hence the modified policy has pathwise holding cost no larger than the original policy.  Repeating this interchange at every violation gives an optimal policy with the stated property.
\end{proof}
\end{chatgptbox}

\subsection{More Prompt Examples and Responses}
After the first prompt, the author's later prompts mainly asked GPT-5.5 Pro to perform several types of tasks: check boundary cases, provide more proof steps, make the arguments more rigorous, clarify some confusions from the authors, or revise the presentation.
These demonstrate typical interactions between the authors and the model: the authors checked the mathematical content, raised specific concerns, and asked the model to refine the proof and the exposition.
We include two examples below.

\subsubsection{Example 1}

In the first example, we asked the model to add more details about a step in the proof of Lemma~ 2.

\begin{promptbox}[title={Prompt that asked for more proof details}]
In the revised manuscript, can you add more details about the following: 
- When you verify monotonicity (K1) in the proof of Lemma 2, can you add a detailed explanation on why it suffices to show H_s(n+1) - H_r(n) >= 0 for each possible pair r <= s? Make sure you consider all the boundary cases for n. 
- Similarly, when you verify Convexity (K2) in the proof of Lemma 2, can you explain in detail why (3.7) is sufficient? 
Add these arguments to the manuscript without changing anything else.
\end{promptbox}

We show below how the model responded to the first bullet point in the above prompt.
\begin{chatgptbox}[title={Relevant part in GPT generated file \emph{before} the prompt above}]
For (K1), it suffices to show $H_s(n+1)-H_r(n)\ge0$ for each possible pair $r\le s$.
\end{chatgptbox}

\begin{chatgptbox}[title={Relevant part in GPT generated file \emph{after} the prompt above}]
For (K1), it suffices to show $H_s(n+1)-H_r(n)\ge0$ for each possible pair $r\le s$.

Here is the reduction in detail.  Fix a tail level $j\in\{0,1,2\}$ and an admissible total population $n\ge j$.  The (K1) inequality for $\bar H$ at this tail level is
\begin{equation*}
        \hspace{5.5cm}
        \bar H_j(n+1)-\bar H_j(n)\ge0 .
        \hspace{5.5cm}
        \text{(3.7a)}
\end{equation*}
By (3.4), if $r=r_j(n)$ and $s=r_j(n+1)$ are the selected indices at $n$ and $n+1$, then
\[
        \bar H_j(n)=H_r(n),\qquad \bar H_j(n+1)=H_s(n+1).
\]
Therefore (3.7a) is exactly
\[
        \hspace{5.5cm}
        H_s(n+1)-H_r(n)\ge0. 
        \hspace{5.5cm}
        \text{(3.7b)}
\]
Since each selector $r_j(\cdot)$ is nondecreasing in total population, every actual pair satisfies $r\le s$.  Conversely, every instance of (K1) for $\bar H$ is obtained in this way from one such adjacent selector pair.  Hence checking (3.7b) for all feasible nondecreasing pairs is sufficient.

The small-population boundary cases are also covered by this reduction.  For $j=0$ and $n=0$, the tail at $n=0$ contains only index $0$, so $r=0$ and $s\in\{0,1\}$; the only possible pairs are $(0,0)$ and $(0,1)$.  For $j=0$ and $n=1$, the first tail contains only indices $0,1$ and the second contains $0,1,2$, so the possible pairs are a subset of $(0,0),(0,1),(0,2),(1,1),(1,2)$.  For $j=0$ and $n\ge2$, all six nondecreasing pairs in the table below can occur.  For $j=1$ and $n=1$, the first tail contains only index $1$, so the only possible pairs are $(1,1)$ and $(1,2)$; for $j=1$ and $n\ge2$, the possible pairs are $(1,1),(1,2),(2,2)$.  For $j=2$, the tail always contains only index $2$, so only $(2,2)$ occurs.  Thus the table checks a superset of the boundary-feasible cases, and every displayed term is defined whenever the corresponding pair can actually arise.
\end{chatgptbox}

\subsubsection{Example 2}

In the second example, we asked the model to add more details for the proof of the average-cost optimality result. 
\begin{promptbox}
Now let's look at Section 7. Let's use the most updated version, ``2026_05_21_One_fast_two_slow_conjecture.pdf" in sources. Section 7 is a bit rough now. I asked an expert to review this section and got the feedback written in ``mz_comments.pdf" in the sources. You can also refer to the paper "Optimal control of a queueing system with two heterogeneous servers" by Lin and Kumar, also in the sources, to see how they handled such arguments. Please rewrite Section 7 to address all the issues raised. You probably want to prove properties of the cone \K. Put those properties in Section 2. Don't change other parts of the paper. Give me the pdf file and latex file.
\end{promptbox}
The feedback mentioned above was given by one of the authors, Matthew Zurek.  In this file, one paragraph is the following:
``The current Section 7 is basically a proof sketch rather than an acceptable proof. Many details need to be added. In particular, I think actual weights for the weighted sup-norm need to be specified/constructed and stability should actually be checked with a Lyapunov argument. Then one can either cite a theorem which helps with taking the limit as $\alpha\downarrow 0$ [justifying that the sequence of value function differences is bounded, has a subsequence with a limit, that limit satisfies the average-cost optimality equation, and also that this optimality equation has a stationary minimizing policy (which is obvious since the action space is finite)].''

Before this prompt, Section 7 was short, with about half a page, and mostly a proof sketch.
In response to this, the model expanded Section~7 to about 3 pages and generated what was closer to an actual proof.
This proof used the results from \citet{Sen_89}, and verified the assumptions.
The proof was roughly correct, but had redundant parts and was hard to read. It was later reorganized and polished by the authors.

\end{document}